\documentclass[12pt]{amsart}%
\usepackage{amsmath}
\usepackage{graphicx}
\usepackage{amscd}
\usepackage{geometry}
\usepackage{amsfonts}
\usepackage{amssymb}
\usepackage[small,nohug,heads=littlevee]{diagrams}
\diagramstyle[labelstyle=\scriptstyle]
\newtheorem{theorem}{Theorem}[section]
\theoremstyle{plain}

\newtheorem{corollary}[theorem]{Corollary}

\newtheorem{definition}{Definition}
\newtheorem{example}{Example}

\newtheorem{lemma}[theorem]{Lemma}

\newtheorem{proposition}[theorem]{Proposition}
\newtheorem{remark}{Remark}

\numberwithin{equation}{section}

\diagramstyle[labelstyle=\scriptstyle]
\begin{document}
\title[Topological Regular Neighborhoods]{Topological Regular Neighborhoods \thanks{This version of Part I is bare in
spots and short on polish, but experts will find all necessary data.}}
\author{Robert D. Edwards}
\address{Department of Mathematics, University of California, Los Angeles, }
\email{rde@math.ucla.edu}
\dedicatory{\vspace{5.25in}\medskip
\_\_\_\_\_\_\_\_\_\_\_\_\_\_\_\_\_\_\_\_\_\_\_\_\_\_\_\_\_\_\_\_\_\_\_\smallskip
\\Preparation of the electronic manuscript was supported by NSF Grant
DMS-0407583. Final editing was carried out by Fredric Ancel, Craig Guilbault
and Gerard Venema.}
\begin{abstract}
This article is one of three highly influential articles on the topology of
manifolds written by Robert D. Edwards in the 1970's but never published.
Organizers of the Workshops in Geometric Topology
(http://www.uwm.edu/\allowbreak\symbol{126}craigg/workshopgtt.htm) with the
support of the National Science Foundation have facilitated the preparation of
electronic versions of these articles to make them publicly available.
Preparation of the first of those articles \textquotedblleft Suspensions of
homology spheres\textquotedblright\ was completed in 2006. A more complete
introduction to the series can be found in that article, which is posted on
the arXiv at: http://arxiv.org\allowbreak/abs/math/\allowbreak0610573v1 and on
a web page devoted to this project: http://www.\allowbreak uwm.\allowbreak
edu/\allowbreak\symbol{126}craigg/\allowbreak EdwardsManuscripts.htm

Preparation of the second article \textquotedblleft Approximating certain
cell-like maps by homeomorphisms\textquotedblright\ is nearing completion. The
current article \textquotedblleft Topological \ regular
neighborhoods\textquotedblright\ is the third and final article of the series.
(\textbf{Note. }This ordering is not chronological, but rather by relative
readiness of the original manuscripts for publication.) It develops a
comprehensive theory of regular neighborhoods of locally flatly embedded
topological manifolds in high dimensional topological manifolds. The following
orignial abstract for that paper was also published as an AMS research
announcement:\medskip

\noindent\textbf{Original Abstract. }(AMS Notices Announcement): A theory of
topological regular neighborhoods is described, which represents the full
analogue in TOP of piecewise linear regular neighborhoods (or block bundles)
in PL. In simplest terms, a topological regular neighborhood of a manifold $M$
locally flatly embedded in a manifold $Q$ ($\partial M=\varnothing=\partial
Q\;$here) is a closed manifold neighborhood $V$ which is homeomorphic fixing
$\partial V\cup M$ to the mapping cylinder of some proper surjection $\partial
V\rightarrow M$. The principal theorem asserts the existence and uniqueness of
such neighborhoods, for $\dim Q\geq6$. One application is that a cell-like
surjection of cell complexes is a simple homotopy equivalence (first proved
for homeomorphisms by Chapman). There is a notion of transversality for such
neighborhoods, and the theory also holds for locally tamely embedded polyhedra
in topological manifolds. This work is a derivative of the work of
Kirby-Siebenmann; its immediate predecessor is Siebenmann's \textquotedblleft
Approximating cellular maps by homeomorphisms\textquotedblright\ Topology
11(1972), 271-294.

This version of Part I is bare in spots and short on polish, but experts will
find all necessary details. Part II is only sketched.

\end{abstract}
\maketitle
\tableofcontents

\noindent\textbf{Note from the editors. }This manuscript is an electronic
version of a handwritten manuscript obtained from the author and dating back
to 1973. As noted in the abstract, this is not a complete and polished work.
Part I is nearly complete but lacking in a few details; a plan for Part II is
described in the manuscript, but there is no evidence it was ever written.
Despite its incomplete nature, the handwritten version of this manuscript was
widely circulated and read. Its influence can be deduced from its appearance
(sometimes under the alternative title \textquotedblleft TOP regular
neighborhoods\textquotedblright) in the bibliographies of a large number of
important papers from that era.

In the process of editing the original manuscript, some obvious `typos' were
corrected and a few other minor improvements were made. For example a number
of missing references, which the author had intended to fill in later, have
been included, and others were updated from preprint status to their final
publication form. (This accounts for a few post-1973 references in the
bibliography.) In a few places, modern notation---more compatible with a \ Tex
doucument---replaces earler notation. Otherwise, this version remains faithful
to the original. In particular, no attempt was made to complete unfinished
portions of the manuscript. Notes from the author (sometimes to himself) about
missing details or planned improvements are included. The decision to leave
the manusript largely unaltered leads to a few awkward situations. For
example, some passages make references to the unwritten `Part II'; and in a
few places there are incomplete sentences---sometimes due to phrases cut off
or rendered unreadable by Xerox machines from long ago. A missing portion of
text is indicated by a short blank line: $\underline{\quad\quad\quad}$.
Despite the minor imperfections, readers will find much interesting and
important mathematics, and some excellent exposition, on these pages.

The editors apologize and accept full responsibility for any new errors that
crept into the manuscript during the conversion process.\newpage

\part{\bigskip}

\section{Introduction\bigskip}

A topological regular neighborhood of a manifold $M$ locally flatly embedded
in a manifold $Q$ ($\partial M=\varnothing=\partial Q$ here; all manifolds
topological) is most easily defined as a closed manifold neighborhood $V$ of
$M$ in $Q$ such that $(V;\partial V,M)$ is homeomorphic to the mapping
cylinder $(Z(r|);\partial V,M)$, of the restriction to $\partial V$ of some
proper retraction $r:V\rightarrow M$. The basic aim of this paper is to prove
the existence and uniqueness of such neighborhoods, for $\dim Q\geq6$. This is
essentially accomplished in Sections 5 and 6. It turns out that such
neighborhoods are more useful if their definition is given in less stringent
form. The alternative (but equivalent) definition is given in Section 1 and
developed in Sections 3 and 4.

\indent Topological regular neighborhoods can be regarded as the analogue in
TOP of block bundles in PL. They have the disadvantage of certain dimension
restrictions, but they have the advantage of a bit more flexibility: certain
pathological fibers are permitted and conversely certain nice fibers can be demanded.

\indent For example, the following is true: if $M^{m}$ is a locally flat
submanifold of $Q^{m+q}$ (say no boundaries), $m+q\geq6$, then $M$ has a
closed manifold mapping cylinder neighborhood $V$ in $Q$ (as above) such that
all fibers $\{r^{-1}(x)\}$ are locally flat $q$-discs which intersect
$\partial V$ in locally flat $(q-1)$-spheres.

\indent Hence one feature of topological regular neighborhoods is that they
may serve as ersatz disc bundle neighborhoods in dimensions where the latter
may fail to exist (see Remark 1.3). However, they have other uses as well, for
example for showing that a cell-like map of cell complexes is a simple
homotopy equivalence, and transversality. The theory also extends to tamely
embedded polyhedra in topological manifolds.

\indent There are several other prior and related neighborhood theories in the
literature, but we defer discussion of these until Section 2, after definitions.

\indent This work grew out of my alternative proof \cite{E2} of Chapman's
Theorem that a topological homeomorphism of polyhedra is a simple homotopy
equivalence. In fact, it was developed to correct a flaw in my first proof of
that theorem, a flaw which it turned out had a much simpler remedy. (The flaw
was an implicit assumption that all triangulations are combinatorial; the
remedy is represented by Theorem 1.2 in \cite{E1}.)

\indent I would like to thank L. Siebenmann for his many valuable comments and
suggestions concerning this paper. Also, I thank Alexis Marin and Ron Stern
for their participation in its development.\newline

\section{Notation, definitions and some examples}

Throughout this paper, we will adhere to the following notational conventions.%
\[
B^{n}=[-1,1]^{n}\subset\mathbb{R}^{n}=\mathbb{R}^{n}\times0\subset
\mathbb{R}^{q}.
\]
$\partial B^{n}$ or $\dot{B}^{n},$ $\operatorname*{int}B^{n}$ or $\mathring
{B}^{n}$, $rB^{n}$ and $r\partial B^{n}$ (for $r>0$) are all used in the usual
ways. $D^{n}$ is used to denote any homeomorphic copy of the unit ball
$\{(x_{1},\ldots,x_{n})\in\mathbb{R}^{n}\mid\sum x_{i}^{2}\leq1\}$ in
$\mathbb{R}^{n}$ and $S^{n-1}$ any homeomorphic copy of its boundary; if the
context requires, regard $D^{n}$ and $S^{n-1}$ as actually being the unit ball
and sphere. (Reason for this rigmarole: Sometimes its useful to have distinct
$n$-cells $B^{n}$ and $D^{n}$ around.)

\indent Given map$\;f:X\rightarrow Y$, let $Z(f)$ denote the mapping cylinder
and $\rho:Z(f)\rightarrow Y$ the mapping cylinder retraction. Thus
\[
Z(f)=(X\times\lbrack0,1]\sqcup Y)/\{(x,1)\sim f(x)\text{ for }x\in X\}
\]
and $\rho(x,t)=f(x)$.

\indent A map of pairs $f:(X,A)\rightarrow(Y,B)$ is \emph{faithful }if
$f^{-1}(B)=A$, not more. We prefer to save `proper' for its more widespread
meaning: $f:X\rightarrow Y$ is \emph{proper} if preimages of compact sets are compact.

\indent The notation $f:X{\normalsize
{\includegraphics[
trim=0.000000in 0.002046in 0.000000in 0.002046in,
height=0.1064in,
width=0.2093in
]%
{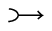}%
}%
}Y$ indicates that $\operatorname*{domain}(f)\subset X$, not necessarily equal
to $X$.

\indent Suppose $M^{m}$ is a topological manifold (with or without boundary,
compact or not). The following definition is the first of two.

\begin{definition}
[Mapping cylinder version]An \textit{(abstract) }\textbf{topological regular
neighborhood} of $M^{m}$ (TRN for short) is a triple $(V^{m+q},M^{m},r)$ where
$V$ is a manifold-with-boundary and $r:V\rightarrow M$ is a proper retraction
such that
\end{definition}

\begin{itemize}
\item[(1)] $(M,\partial M) \hookrightarrow(V, \partial V)$ is a faithful,
locally flat inclusion (\textit{faithful} $\equiv M \cap\partial V = \partial
M)$,

\item[(2)] $\delta V\equiv r^{-1}(\partial M)$ is a collared codimension 0
submanifold of $\partial V$ (define $\dot{V}=cl(\partial V-\delta V)$ and
$\mathring{V}=V-\dot{V}$), and

\item[(3)] $(V;\dot{V},M,r)$ is isomorphic (keeping $\dot{V}\cup M$ fixed) to
the mapping cylinder of $r|:\dot{V}\rightarrow M$, that is, $(V;\dot
{V},M,r)\approxeq(Z(r|_{\dot{V}});\dot{V},M,\rho)$ where $\rho$ is the mapping
cylinder retraction
\end{itemize}

\indent This definition, although quite natural, turns out to be too
restrictive for certain purposes. For example, one would like the composition
of TRN's to be a TRN. Consider:

\begin{example}
\label{Ex 1.1}(See Figure 1.) This example describes two mapping cylinder
TRN's $r_{1}:V_{1}\rightarrow V_{2}$ and $r_{2}:V_{2}\rightarrow J$ whose
composite $r_{2}r_{1}:V_{1}\rightarrow J$ is not a mapping cylinder TRN.

For the purposes of this example, let $I=J=K=[-1,1]$, to be thought of as
first, second, and third coordinate intervals in $\mathbb{R}^{3}$.

Let $(J\times K,J,r_{0})$ be the mapping cylinder TRN as pictured in Figure
1a,\begin{figure}[th]
\begin{center}
\includegraphics{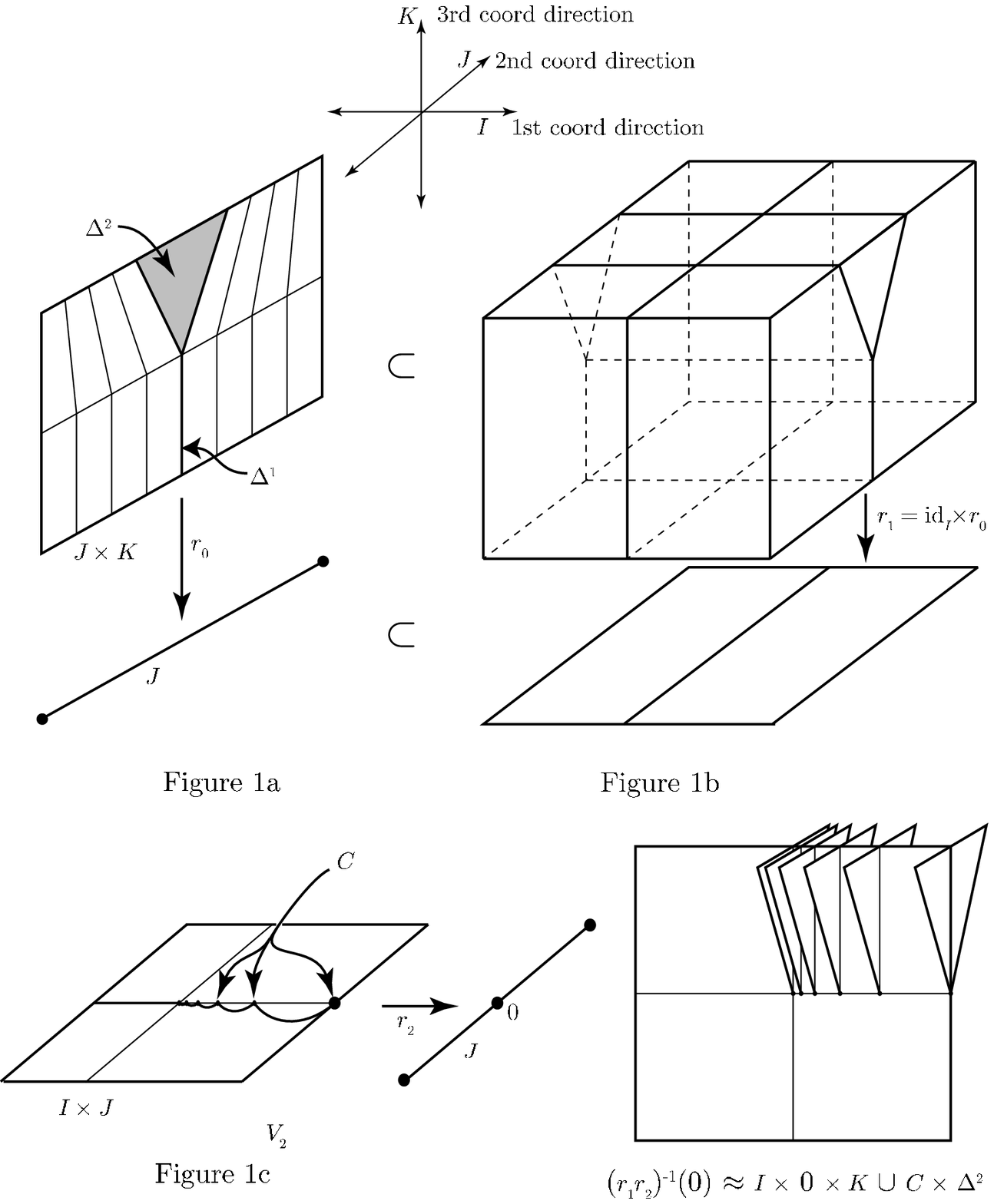}
\end{center}
\caption{Example 1: Composition of TRN's}%
\end{figure}such that $r_{0}^{-1}\left(  0\right)  =\Delta^{1}\vee\Delta^{2}$
is the only non-interval point inverse. Product this TRN with the interval $I$
to get
\[
(V_{1},V_{2},r_{1})\equiv(I\times J\times K,I\times J,\operatorname{id}%
_{I}\times r_{0})
\]
as shown in Figure 1b.

Let $V_{2}=I\times J\overset{r_{2}}{\longrightarrow}J$ be obtained from the
standard-projection TRN $\pi_{J}:I\times J\rightarrow J$ by a slight
perturbation of the projection map $\pi_{J}$, as shown in Figure 1c.
Specifically, let $h:I\times J\rightarrow I\times J$ be a $(t\times J)$-level
preserving homeomorphism such that%
\[
h(I\times0)\cap I\times0=C\equiv\operatorname*{cl}\{(1/n,0)\mid n>0\}\subset
I\times0\text{,}%
\]
and define $r_{2}=\pi_{J}h^{-1}:I\times J\rightarrow J$.

The composition $r_{2}r_{1}:V_{1}\rightarrow J$ is not a mapping cylinder
projection, as $(r_{2}r_{1})^{-1}(0)\approx(I\times0\times K)\cup
(C\times\Delta^{2})$, which is not a cone. See Figure 1d.
\end{example}

Before giving the second definition, it is worth considering the analogous
situation in PL for motivation. There, one can define an abstract regular
neighborhood of a manifold $M$ (without boundary here) as a triple $(V,M,r)$
where $V$ is a manifold with boundary, $M\subset\operatorname*{int}V$ and
$r:V\rightarrow M$ is a PL collapsible retraction, where \emph{collapsible}
means each point inverse $r^{-1}(x)$ is a collapsible polyhdedron. This is M.
Cohen's observation \cite{Co2}, and it provides an alternative way of defining
block bundles \cite[\S 4]{RSI}. Cohen shows that such a $V$ has topological
mapping cylinder structure (\cite{Co1}; one has to be careful with PL mapping
cylinders [?]). With this definition, the composition of PL regular
neighborhoods, as defined above, is readily a PL regular neighborhood
\cite[Lemma 8.6]{Co2}.

\indent The most general analogue in TOP of a piecewise linear collapsible
polyhedron is a cell-like compactum. This suggests the topological adaptation
of Cohen's definition. First we need some preliminary definitions, which we
give in anodyne form for the nonexpert in shape theory.

\indent Let $X$ be a finite dimensional compact metric space. Such an $X$ is
\emph{cell-like} \label{cell-like}if $X$ embeds in some euclidean space
$\mathbb{R}^{q}$ so that its image is cellular, that is, the intersection of
open $q$-cells. Similarly, $X$ is $k$\emph{-sphere-like} if $X$ embeds in some
euclidean space $X\hookrightarrow\mathbb{R}^{q}$ so that $X=\cap_{i=1}%
^{\infty}f_{i}(S^{k}\times\mathbb{R}^{q-k})$ where each $f_{i}:S^{k}%
\times\mathbb{R}^{q-k}\rightarrow\mathbb{R}^{q}$ is an embedding, and
$\operatorname*{image}f_{i+1}\hookrightarrow\operatorname*{image}f_{i}$ is a
homotopy equivalence. Also, $X$ is $k$\emph{-UV} if given any embedding
$X\hookrightarrow\mathbb{R}^{q}$ and any neighborhood $U$ of $X$ in
$\mathbb{R}^{q}$, there is a neighborhood $V$ of $X$, $V\subset U$, such that
any map $\alpha:S^{k}\rightarrow V$ is null-homotopic in $U$. $X$ is
$\emph{UV}^{k}$ if it is $j$-\emph{UV} for $0\leq j\leq k$. It is the message
of shape theory that these properties are intrinsic properties of $X$ and can
be so characterized, without any reference to a specific embedding.

\indent An inclusion $Y\hookrightarrow X$ of a closed subset $Y$ into a
locally compact, finite-dimensional separable metric space $X$ is a
\emph{shape equivalence} if for any embedding $X\hookrightarrow Q$ of $X$ onto
a closed subset of a manifold $Q$ (i.e., proper embedding), the following
holds: given neighborhoods $U$ of $X$ in $Q$ and $W$ of $Y$ in $Q$, there is a
neighborhood $V$ of $X$ in $Q$ such that $V$ homotopically deforms into $W$ in
$U$, keeping some neighborhood $N$ of $Y$ fixed. That is, there is a homotopy
$h_{t}:V\rightarrow U$, $t\in\lbrack0,1]$, joining $\operatorname{id}%
_{V}=h_{0}:V\rightarrow U$ to a map $h_{1}:V\rightarrow W$, such that
$h_{t}|_{N}=\operatorname{id}$. (independent of $t$).%
\begin{figure}[th]
\begin{center}
\includegraphics{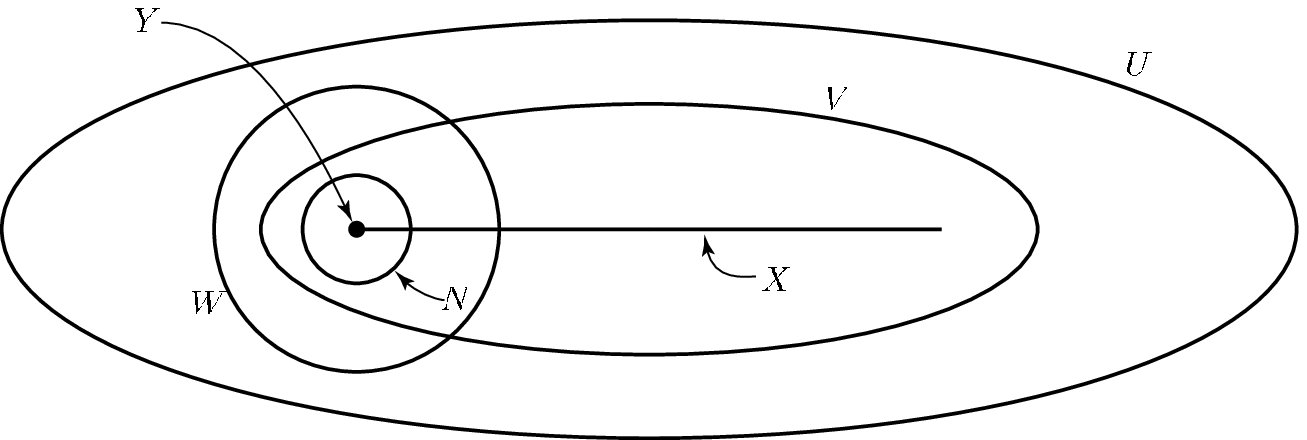}
\end{center}
\end{figure}

Shape theory says that if this definition holds for one proper embedding
$X\hookrightarrow Q$, it hold for all \cite[p.499]{La2}.

Suppose $r:V\rightarrow M$ is a proper retraction of spaces and $\dot{V}$ is
some distinguished closed subset of $V$. We use the notation $F_{x}=r^{-1}(x)$
and $\dot{F}_{x}=F_{x}\cap\dot{V}$, for $x\in M$. Recall $r$ is
\emph{cell-like} (CE) if each $F_{x}$ is cell-like \cite{La}. We call $r$
\emph{cell-like, sphere-like} (CS) if each $F_{x}$ is cell-like and each
$\dot{F}_{x}$ is sphere-like. Finally (and most importantly), we call $r$
\emph{cone-like} if each $F_{x}$ is cell-like and the pair $(F_{x}-x,\dot
{F}_{x})$ is proper shape equivalent to $(\dot{F}_{x}\times\lbrack0,1),\dot
{F}_{x})$. See \cite{BS}. In order to obviate proper shape theory, we remark
in advance that in the following definition, one can interpret
\emph{cone-like} \label{cone-like}to mean that $r$ is CS and each inclusion
$\dot{F}_{x}\hookrightarrow F_{x}-x$ is a shape equivalence (in fact, in
codimension $\geq3$ one need only assume $r$ is CE and each $\dot{F}_{x}$ has
property $1$-UV; details are in \S 3).

\indent$M^{m}$ is a topological manifold (with or without boundary, compact or
not). The following definition differs from the previous mapping cylinder
version only in condition (3).

\begin{definition}
[Cone-like version]An (abstract)\textbf{ topological regular}
\textbf{neighborhood }of $M^{m}$ is a triple $(V^{m+q},M^{m},r)$ where $V$ is
a manifold-with-boundary and $r:V\rightarrow M$ is a proper retraction such that
\end{definition}

\begin{itemize}
\item[(1)] $(M,\partial M)\hookrightarrow(V,\partial V)$ is a faithful locally
flat inclusion.

\item[(2)] $\delta V\equiv r^{-1}(\partial M)$ is a collared codimension $0$
submanifold of $\partial V$ (define $\dot{V}=cl(\partial V-\delta V)$ and
$\mathring{V}=V-\dot{V})$, and

\item[(3)] $r:V\rightarrow M$ is cone-like.\footnote{See the note at bottom of
page \pageref{condition(0)}.}
\end{itemize}

\indent The following examples are to illuminate the definition. The last two
are relevant only to codimension 2.\medskip

\begin{example}
This example shows why $r$ must be more than just cell-like. Let $V$ be any
compact contractible manifold and $m=\operatorname{point}\;\mathbb{\in
}\operatorname*{int}V$ and $r:V\rightarrow m$ the retraction. Then $r$ is CE,
but if one wants uniqueness to hold in the theory, there must be some
condition which force $\partial V$ to be a homotopy sphere instead of just a
homology sphere.
\end{example}

\begin{example}
This shows the need for the strong cone-like hypothesis on $r$ in codimension
2. (For polyhedra, see Siebenmann's example in \S 8\footnote{\textbf{Note from
editors:} In fact, this example did not make it into \S 8.}). Let
$(B^{m+2},D^{m})$ be a knotted locally flat ball pair such that the sphere
pair $(\partial B^{m+2},\partial D^{m})=(\partial B^{m+2},\partial B^{m})$ is
standard. Recall that these can be constructed with $(B^{m+2}-D^{m},\partial
B^{m+2}-D^{m})$ highly connected \cite{Wa}. There is a CS retraction
$r:B^{m+2}\rightarrow D^{m}$ which is a standard $B^{2}$-fibered projection
over $D^{m}-0$ and such that $F_{0}$ is homotopy equivalent to the
contractible space $B^{m+2}-(D^{m}-0)$, with $\dot{F}_{0}\approx S^{1}\times
B^{m}$. Since $(B^{m+2},D^{m})$ is not standard, it is necessary to rule out
such an $r$.
\end{example}

\begin{example}
This shows that in codimension 2, it is not enough to just assume that $r$ is
cell-like and each inclusion $\dot{F}_{x}\hookrightarrow F_{x}-x$ is a shape
equivalence (as opposed to the proper shape equivalence in the definition of
cone-like). \textbf{Note. }This example is incomplete. It requires a knotted
embedding $f:S^{n}\rightarrow S^{n+2}$ which permits a concordance
$F:S^{n}\times I\rightarrow S^{n+2}\times I$ to the standard $S^{2}$ so that
$S^{n+2}-f\left(  S^{n}\right)  \hookrightarrow S^{n+2}\times I-F\left(
S^{n}\times I\right)  $ is a homotopy equivalence (everything locally flat).
Then we could construct this example.\label{ex: cell-like/cone-like}
\end{example}

\begin{remark}
(Concerning $\delta V$). If $(V,M,r)$ is a TRN of $M$, then $(\delta
V,\partial M,r|_{\delta V})$ is a TRN of $\partial M$ (either definition).
\textbf{Note:} $(\delta V)^{\cdot}=\partial{\dot{V}}$, which we will denote
$\delta\dot{V}$; also $(\delta V)^{\circ}=\partial\mathring{V}$, which we will
denote $\delta\mathring{V}$. Actually, our definition of TRN for manifolds
with boundary is not the most general, as one need not require $\delta V$ to
coincide with $r^{-1}(\partial M)$. We postpone this relaxation and its
details until the discussion of neighborhoods of polyhedral pairs in Part II,
where it becomes necessary.
\end{remark}

\begin{remark}
(Concerning the equivalence of definitions). It is routine to show that a
mapping cylinder TRN is a cone-like TRN, using definitions. The converse of
course is not strictly true, but it is as true as could be expected: if
$(V,M,r)$ is a cone-like TRN, then there is a mapping cylinder retraction
$r^{\prime}:V\rightarrow M$ which is arbitrarily close to $r$ and agrees with
$r$ on $\dot{V}$ $(\dim V\neq4;$ for $\dim V=3$ see next remark). That is,
$(V;\dot{V},M,r^{\prime})\approx(Z(r|_{\dot{V}});\dot{V},M,\rho)$ $(rel$
$\dot{V}\cup M)$. Details are in Section 4.
\end{remark}

\begin{remark}
(Concerning non-locally flat embeddings of $M$). The definitions make perfect
sense even if $M$ is not locally flatly embedded in $V$. However, we cannot
say anything non-trivial regarding existence-uniqueness in this case, and the
techniques of this paper are no help there. Recall that if non-combinatorial
triangulations of topological manifolds exist, i.e., if the double suspension
of some genuine homology sphere is topologically homeomorphic to a real sphere
, then there is a nonlocally flat embedding of $S^{1}$ (namely the suspension
circle of the above suspension) into some sphere such that the embedding has a
manifold mapping cylinder neighborhood. Further details are in \cite{Gl}.

If $M^{m}$ is an arbitrary, possibly wild submanifold of $Q^{m+q}$, then
$M=M\times0\subset Q\times\mathbb{R}^{1}$ is locally flat (no dimension
restrictions; details recounted in \cite{BrS} for $q>1$.). Thus if $V$ is a
TRN of a non-locally flatly embedded $M$ (either definition), then
$V\times\lbrack-1,1]$ is a genuine TRN of $M\times0$.
\end{remark}

\begin{remark}
(Concerning disc bundle neighborhoods). Topological regular neighborhoods may
serve as a partial substitute for topological disc bundle neighborhoods in
dimensions where the latter don't exist (although even when disc bundle
neighborhoods exist, the uniqueness of TRN's is still useful; e.g., the
topological invariance of simple homotopy type for cell complexes, \S 9). We
recall what is known about existence-uniqueness of disc bundle neighborhoods.
If $M^{m}\hookrightarrow\operatorname*{int}Q^{m+q}$ is a locally flat
topological embedding, then $M^{m}$ has a unique disc bundle neighborhood if
$m+q\leq3$ (semi-classical); $q=1$ \cite{Bro}, $q=2$, $m+q\geq5$ \cite[AMS
Notices 1971]{KS}, $m\geq3$, $m+q=5,6$ again essentially by \cite{KS} (no
upper bound on $m+q$ for existence); $m\leq q+2$ [resp. $m\leq6$, $m\leq5$],
$q\geq7$ [resp. $q=6$, $q=5$], with existence holding for these $m$ increased
by one \cite{St}.

Hence the first $m+q\neq4$ case where existence fails is $(m,m+q)=(4,7)$,
realizable by a counterexample of Hirsch.
\end{remark}

\begin{remark}
(Concerning low dimensions). Subsequent theorems in Part I are all stated and
proved for ambient dimension $\geq6$ (exceptions: the mapping cylinder theorem
(\S 4) only requires ambient dimension $\geq5$, and the local contractibility
theorem (\S 7) has no dimension restrictions). As usual, all theorems hold
when ambient dimension $\leq2$ (same proofs work) and all theorems hold when
ambient dimension $=3$, if we adopt the same convention that Siebenmann did in
\cite{Si1} to get around the Poincar\'{e} conjecture: in the cone-like
definition of TRN, assume in addition that each fiber $F_{x}$ has a manifold
neighborhood in $V$ which is prime $(\equiv$ there is no 2-sphere which
separates the manifold into two non-cells). The mapping cylinder definition
works as stated; its fibers automatically have this property. Ambient
dimensions 4 and 5 remain a mystery because of the failure of the s-cobordism
theorem there \cite{Si4}. But remember that in dimension 5, disc bundle
neighborhoods exist and are unique (see preceding Remark).
\end{remark}

\indent We continue with more definitions. Two abstract topological regular
neighborhoods $(V_{0},M,r_{0})$ and $(V_{1},M,r_{1})$ are \emph{homeomorphic}
if they are homeomorphic as triples $(V_{0},M,\delta V_{0})\approx
(V_{1},M,\delta V_{1})$ keeping $M$ fixed. Two such TRN's are
\emph{isomorphic} if they are homeomorphic via $h:(V_{0},M,\delta
V_{0})\overset{\approx}{\longrightarrow}(V_{1},M,\delta V_{1})$ so that
$r_{1}=r_{0}h^{-1}$. This notion seldom arises because of its excessive strength.

\indent If $(M,\partial M)\hookrightarrow(Q,\partial Q)$ is a faithful locally
flat inclusion and $V$ is a TRN of $M$ in $Q$, we always assume (unless
otherwise stated) that $V\cap\partial Q=\delta V$ and that $(\dot{V}%
,\delta\dot{V})$ is collared in $(Q-\mathring{V},\partial Q-\delta\mathring
{V})$. Two TRN's $(V_{0},M,r_{0})$ and $(V_{1},M,r_{1})$ of $M$ in $Q$ are
\emph{equivalent }in $Q$ if there is a homeomorphism of $Q$ whose restriction
gives a homeomorphism of $V_{0}$ onto $V_{1}$. They are \emph{equivalent by
ambient isotopy} if this homeomorphism can be chosen isotopic to
$\operatorname{id}_{Q}$ through homeomorphisms of $Q$ fixed on $M$. Invariably
such an ambient isotopy will by construction leave a neighborhood of $M$
fixed; if not, it can be so arranged by the isotopy extension theorem.

\indent Although not explicitly required in the definition, all our
equivalences by ambient isotopy $h_{t}:Q\rightarrow Q$, $t\in\lbrack0,1]$, can
be followed by a \emph{cone-like homotopy} $r_{t}^{\prime}:V_{1}\rightarrow M$
($\equiv$ homotopy through cone-like retractions) joining $r_{0}^{\prime
}=r_{0}h_{1}^{-1}$ to $r_{1}^{\prime}=r_{1}$. This will sometimes prove
useful, and will be mentioned explicitly whenever it arises.

\indent We conclude this section with a useful example, which captures the
difference between topological disc bundles and topological regular neighborhoods.

\begin{example}
\label{Ex: capping off}(Capping Off). This example illustrates the fundamental
compactification operation for TRN's. Suppose $r:\mathbb{R}^{m}\times
B^{q}\rightarrow\mathbb{R}^{m}=\mathbb{R}^{m}\times0$ is any cone-like
retraction. Regard $S^{m}=\mathbb{R}^{m}\cup\infty$ and define $i$ =
$\operatorname{inclusion}\times\operatorname{id}:\mathbb{R}^{m}\times
B^{q}\hookrightarrow S^{m}\times B^{q}$. Then $\overline{r}=S^{m}\times
B^{q}\rightarrow S^{m}$ defined by
\[
\overline{r}=\left\{
\begin{array}
[c]{ll}%
iri^{-1}\quad\text{on}\quad(S^{m}-\infty)\times B^{q} & \\
\operatorname{projection}\text{ to}\;\infty\;\text{on}\;\infty\times B^{q} &
\end{array}
\right.
\]
is a cone-like retraction.
\begin{figure}[th]
\begin{center}
\includegraphics{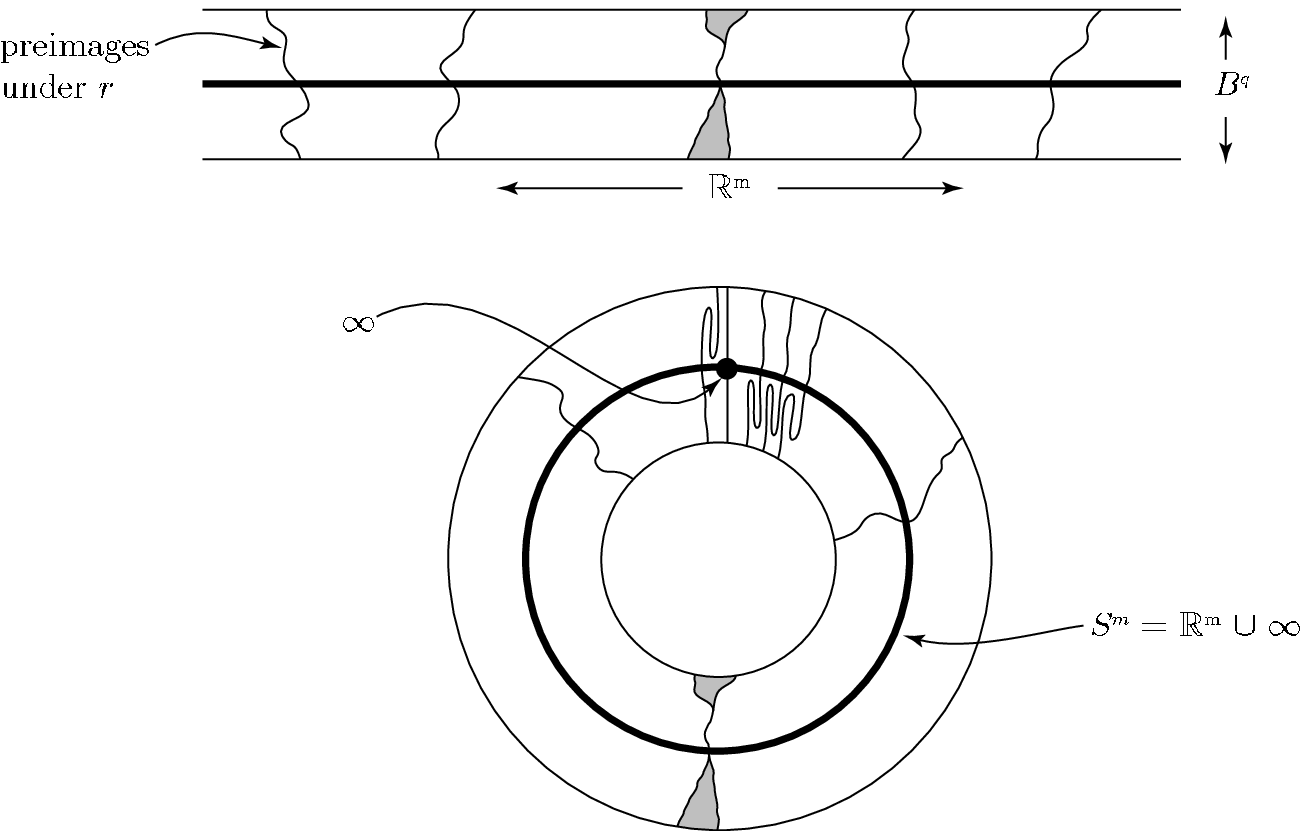}
\end{center}
\end{figure}

\end{example}

\section{Statement of results; general remarks}

The primary goal of Part I is to prove:

\begin{theorem}
[Existence-Uniqueness Theorem]\label{existence-uniqueness thm}Suppose
$(M^{m},\partial M)\hookrightarrow(Q^{m+q},\partial Q)$ is a faithful locally
flat inclusion of topological manifolds, $m+q\geq7$, $(\geq6$ provided that
$\partial M=\varnothing$ or that the conclusion already holds at $\partial
M)$. Then $M$ has a topological regular neighborhood in $Q$, and any two are
equivalent by ambient isotopy of $Q$.
\end{theorem}

\noindent\textbf{Addendum.} The ambient isotopy $h_{t}:Q\rightarrow Q$ which
realizes the homeomorphism of $(V_{0},M,r_{0})$ to $(V_{1},M,r_{1})$ may be
chosen as the composition $h_{t}=h_{1,t}^{-1}h_{0,t}$ of two well-controlled
ambient isotopies $h_{0,t}$ and $h_{1,t}$, where \emph{well-controlled} means
that each $h_{i,t}$, $t\in\left[  0,1\right]  $, moves only those points which
lie near $V_{i}$ but not near $M$, along tracks which lie arbitrarily close to
individual fibers of $V_{i}$. Furthermore, the cone-like retractions
$r_{0}h_{1}^{-1}$ and $r_{1}$ of $V_{1}$ to $M$, which are close by
construction, can be joined by a small cone-like homotopy.\newline%
\newline\textbf{Remark 2.2} (Concerning special neighborhoods.) There are
actually several useful subclasses of TRN's, each gotten by putting more
restrictions on the fibers $(F_{x},\dot{F}_{x})$ in either original
definition. The Existence-Uniqueness Theorem holds for each class (with no
change in the proof). Some sub-classes are in order of increasing restrictiveness:

\begin{itemize}
\item[(1)] (the original fibers, for comparison) $(F_{x},\dot{F}_{x}%
)\overset{\text{shape}}{\thicksim}(B^{q},S^{q-1})$

\item[(2)] $(F_{x},\dot{F}_{x})\overset{\text{htpy equiv.}}{\thicksim}%
(B^{q},S^{q-1})$

\item[(3)] (1) plus $F_{x}$ and $\dot{F}_{x}$ are ANR's. Note this implies (2) holds.

\item[(4)] $(F_{x},\dot{F}_{x})\overset{\text{homeo.}}{\approx}(B^{q}%
,S^{q-1})$

\item[(5)] $(F_{x},\dot{F}_{x})\overset{\text{homeo.}}{\approx}(B^{q}%
,S^{q-1})$ and each $(F_{x},\dot{F}_{x})$ is locally flat in $(V,\dot{V})$.
\end{itemize}

\indent Class (5) provides the nicest neighborhoods as far as existence is
concerned, whereas the original cone-like definition offers the strongest
uniqueness theorem. The cone-like homotopy of the Addendum belongs to the
appropriate class.

\indent The theory of topological regular neighborhoods is quite evidently
modelled on the theory of PL regular neighborhoods and PL block bundles (which
are really the same things, looked at from different perspectives, c.f.
\cite[\S 4]{RSI}. For the former, our preferred reference is Cohen \cite{Co2},
and we have already remarked (in \S 1 after Example \ref{Ex 1.1}) how the
treatment there is reflected here. Topological regular neighborhoods are not
by definition partitioned into blocks, but they can be if the core manifold
$M$ has a handle structure (as it does if $\dim M\neq4,5)$. This is discussed
more fully in Part II. Topological regular neighborhood theory is completely
parallel to block bundle theory, except for the bothersome dimension restrictions.

\indent It is worth recalling other topological neighborhood theories which
are already established. Suppose $X$ is a compact subset of a topological
manifold $Q$. If $X$ is arbitrary there is little that can be said, except
that most embeddings of $X$ into $Q$ (most $\equiv$ a dense $G_{\delta}$
subset of all embeddings) are \emph{locally tame}, defined to mean $Q-X$ is
$k$-LC at $X$ for all $0\leq k\leq\dim Q-\dim X-2$, where $\dim X$ is the
covering dimension. Interestingly, in the trivial range $2\dim X+2\leq\dim
Q\neq4$, homotopy implies ambient isotopy for such locally tame embeddings
\cite{Bry}. Below this range there is no hope of classifying neighborhoods as
there may be uncountably many distinct neighborhood germs, even for $X$ a
locally tamely embedded ANR.

If $X$ is shape dominated by a finite complex, there is a nice theory of open
regular neighborhoods worked out by Siebenmann \cite{Si3}. Briefly, an open
regular neighborhood of $X$ in $Q$ is an open neighborhood $U$ which satisfies
a certain compression property: given any compact subset $K$ of $U$ and any
neighborhood $W$ of $X$, there is a homeomorphism $h$ of $U$ having compact
support and fixing a neighborhood of $X$, such that $h(K)\subset W$. Such
neighborhoods have the homotopy type of $X$ and are unique. They exist if and
only if $X$ is shape dominated by a finite complex, the \textquotedblleft
if\textquotedblright\ part assuming $\dim X\leq\dim Q-3$ and $X\hookrightarrow
Q$ locally tame. Furthermore $X$ has an open radial neighborhood if and only
if $X$ actually has the shape of a finite complex $(U$ is \emph{radial} if
$U-X\approx Y\times\mathbb{R}^{1}$ for some compactum $Y$). The difference
between these situations is precisely measured by an obstruction in
$\widetilde{K}_{0}(\pi_{1}(U-X))$ that takes arbitrary values.

\indent Johnson has recently observed these facts for $X$ a topological
manifold \cite{Jo}.

\indent If $X^{m}$ is a polyhedron embedded in a topological manifold
$Q^{m+q}$, $q\geq3$, Weller has observed that any two closed manifold
neighborhoods of $X$ which are PL regular neighborhoods in some (possibly
unrelated) PL structures, are topological homeomorphic by Chapman's
topological invariance of simple homotopy type.

\indent This theory of topological regular neighborhoods represents a
sharpened form of the topological regular neighborhood theory of
Rourke-Sanderson \cite{RS4}. Briefly the relation is this: given a fixed
manifold $M$, the Rourke-Sanderson paper classifies germs at $M$ of all
manifold pairs $(Q,M)$, where $M$ is embedded in $Q$ as a locally flat
submanifold; two such pairs $(Q_{0},M)$ and $(Q_{1},M)$ have equivalent germs
if there are neighborhoods $U_{i}$ of $M$ in $Q_{i}$, $i=0,1$, such that
$(U_{0},M)\approx(U_{1},M)$ keeping $M$ fixed. This paper shows that each germ
class $[(Q,M)]$ contains as a representative a unique topological regular
neighborhood $(V,M)$. This paper recovers all the results of \cite{RS4}. We
recall them as they arise.

\indent A word on cell-like maps. They clearly play a central role in this
paper, so it is worth repeating some history from \cite{Si1} (whose complete
introduction is well worth reading). In 1967, D. Sullivan observed that the
geometrical formalism used by S. P. Novikov to prove that a homeomorphism
$h:M\rightarrow N$ of manifolds preserves rational Pontrjagin classes, uses
only the fact that $h$ is proper, and a \emph{hereditary homotopy equivalence}
in the sense that for each open $V\subset N$ the restriction $h^{-1}%
V\rightarrow V$ is a homotopy equivalence. Lacher \cite{La} was able to
identify such proper equivalences as precisely CE maps, providing one
restricts attention to ENR's (= euclidean neighborhood retracts = retracts of
open subsets of euclidean space).

\indent This paper can be regarded as an extension of Siebenmann's \cite{Si1}
in the following sense: he establishes that a cell-like surjection of
$n$-manifolds is a limit of homeomorphisms. This paper establishes that a
cone-like retraction $r:V\rightarrow M$ of manifolds is locally the limit of
disc bundle projections. For this reasons our proofs in \S 5 bear strong
resemblance to Siebenmann's proofs.

\section{Homotopy properties of TRN's}

The purpose of this section is to prove Proposition 3.1. below, which
establishes certain basic homotopy properties of TRN's. The essential result,
without refinements, is that the difference $V_{1}-\mathring{V}_{0}$ between
two TRN's of the same manifold $M\subset V_{0}\subset\mathring{V}_{1}\subset
V_{1}$ is a proper $h$-cobordism.

\indent For simplicity, we will always assume $\partial M=\varnothing=\delta
V$ in this section, with the understanding that the $\partial M\neq
\varnothing\neq\delta V$ versions of all results also hold.

\indent When reading the following Proposition, it is worth keeping in mind
that parts (1) and (2) are trivial for mapping cylinder TRN's.\footnote{This
section, as well as perhaps pages \pageref{cell-like}-\pageref{cone-like},
could have benefitted from an overhaul for clarity. I wish to emphasize that
in Proposition \ref{Prop (homotopy proposition)} what we really want is a
property (0)\label{condition(0)} from which (1) and (2) follow.\medskip
\par
(0) $\left(  V-M,\dot{V}\right)  $ \emph{is proper homotopy equivalent to
}$\left(  \dot{V}\times\lbrack0,1),\dot{V}\times0\right)  $\emph{, by an
}$\varepsilon$\emph{-controlled proper homotopy equivalence.\medskip}
\par
\noindent This property (0) is what \textquotedblleft
cone-like\textquotedblright\ is all about.}\label{footnote cross reference}

\begin{proposition}
[Homotopy Proposition]\label{Prop (homotopy proposition)}Suppose $(V,M,r)$ is
a topological regular neighborhood (either definition). Then

\begin{enumerate}
\item $M$ is a strong deformation retract of $V$. In fact, the following type
of partial deformations exist: Given any majorant map $\epsilon:M\rightarrow
(0,\infty)$ and any neighborhood $U$ of $M$ in $V$, there is a neighborhood
$W$ of $M$, $W\subset U$, and a deformation $f_{t}:V\rightarrow V$,
$t\in\lbrack0,1]$, such that $f_{0}=\operatorname{id}_{V}$, $f_{1}(V)\subset
U$ and for each $t$, $f_{t}|_{W}=\operatorname{id}_{W}$ and $f_{t}(V-W)\subset
V-W$, (i.e., $W$ is `undisturbed' by the homotopy), and $rf_{t}$ is $\epsilon
$-close to $r$.

\item $\dot{V}$ is a strong deformation retract of $V-M$. In fact, given
$\epsilon:M\rightarrow(0,\infty)$, there is a deformation $g_{t}%
:V-M\rightarrow V-M\;(rel\;\dot{V})$, joining $g_{0}=\operatorname{id}_{V-M}$
to a retraction $g_{1}:V-M\rightarrow\dot{V}$, such that for each $t,rg_{t}$
is $\epsilon$-close to $r$.

\item If $(V_{0},M,r_{0})$ is a TRN such that $V_{0}\subset\mathring{V}$ is a
closed neighborhood of $M$ in $V$, then the difference $(V-\mathring{V}%
_{0};\dot{V}_{0},\dot{V})$ is a proper $h$-cobordism.
\end{enumerate}
\end{proposition}

\indent Part (3) is a straightforward consequence of parts (1) and (2). The
remainder of this section is concerned with proving parts (1) and (2) for
cone-like TRN's.

\indent Before proceeding to the proof, we make some brief asides. The first
is to point out that in the definition of cone-like TRN, if one only assumes
that $r:V\rightarrow M$ is CE instead of conelike, then the fibers $F_{x}$ and
their boundaries $\dot{F}_{x}$ all have the cohomology properties one would
expect. Namely, by duality, $\check{H}^{\ast}(\dot{F}_{x})\approx H^{\ast
}(S^{q-1})$ and $\check{H}^{\ast}(F_{x}-x,\dot{F}_{x})=0$ (here $\check
{H}^{\ast}$ denotes \v{C}ech cohomology). See details below. Also in
codimension $\geq3$, $F_{x}-x$ is $1$-UV. However, as Example 2 shows,
$\dot{F}_{x}$ may not have the shape of $S^{q-1}$.

\indent If one is only interested in establishing the non-proper, codimension
$\geq3$ case of part (3) above, there is an especially simple proof, called to
my attention by Alexis Marin.\medskip

\begin{proposition}
[\textbf{Illustrative Proposition}]Suppose $(V_{i}^{m+q},M^{m},r_{i})$,
$i=0,1$, are \textbf{cell-like} TRN's of $M$, with $V_{0}\subset V_{1}$ and
$q\geq3$, such that all fiber boundaries $\{\dot{F}_{x,i}=r_{i}^{-1}%
(x)\cap\dot{V}_{i}\mid x\in M$, $i=0,1\}$ are $1$-UV. Then the inclusion
$\dot{V}_{0}\hookrightarrow V_{1}-M$ is a homotopy equivalence. Hence, if
$V_{0}\subset\mathring{V}_{1}$, the difference $(V_{1}-\mathring{V}_{0}%
;\dot{V}_{0},\dot{V}_{1})$ is an $h$-cobordism (using the additional parallel
facts that $\dot{V}_{i}\hookrightarrow V_{i}-M$ are homotopy equivalences,
$i=0,1$).
\end{proposition}

\begin{proof}
The cell-like retraction $r_{i}:V_{i}\rightarrow M$ is a homotopy equivalence
by the theorem of Lacher.

The maps $V_{i}-M\hookrightarrow V_{i}$ and $\dot{V}_{i}\overset{\alpha
}{\hookrightarrow}V_{i}\overset{r}{\longrightarrow}M$ induce $\pi_{1}%
$-isomorphisms, the first by general position and the others because $r$ and
$r\alpha$ are $1$-UV surjections \cite[p.505]{La2}. Hence all universal covers
are compatible, and we have covering TRN's
\begin{figure}[th]
\begin{center}
\includegraphics{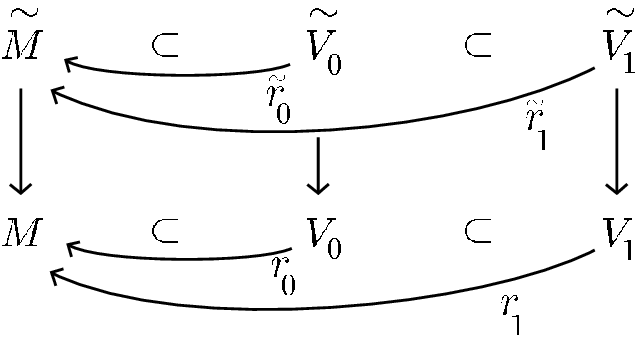}
\end{center}
\end{figure}

It suffices to show $\widetilde{\dot{V}}\hookrightarrow\widetilde{V}%
_{1}-\widetilde{M}$ induces homology isomorphisms, for then the theorems of
Hurewicz and Whitehead apply. The topmost square below represents
Lefschetz/Alexander duality and its naturality, and the remaining squares are
the homology sequence of a pair. (\textbf{Note. }For simplicity, the $\sim$'s
are omitted from the diagram.)%
\[
\begin{CD}
H_{v}^{m+q-\ast}(V_{0})                 @<{\approx}<\text{inclusion}^{\ast}<   H_{c}^{m+q-\ast}(M)\\
@V\text{(Lefschetz Duality)} V{\approx}V                  @V{\approx}V\text{(Alexander Duality)}V\\
H_{\ast}(V_0,\dot{V_{0}})               @>{\approx}>>   H_{\ast}(V_1,V_{1}-M)\\
@V{\partial}VV                                          @V{\partial}VV\\
H_{\ast-1}(\dot{V_{0}})               @>{\approx}>\text{(Five Lemma)}>   H_{\ast-1}(V_{1}-M)\\
@VVV                                                                    @VVV\\
H_{\ast-1}(V_{0})               @>{\approx}>\text{inclusion}_{\ast}>   H_{\ast-1}(V_{1})
\end{CD}
\]

\end{proof}

\indent Unfortunately the above proof has no straightforward generalization to
the proper category and to codimension 2, and it provides no information about
the tracks of the homotopies. For this reason we adopt the following approach,
which is, in a sense, more elementary because it uses no algebra and duality,
but unfortunately is more elaborate, using elementary shape theory.

\indent The following discussion uses the notion of \emph{resolution} of a TRN
$r:V\rightarrow M$, which provides a way of compactifying deleted fibers
$\{F_{x}-x\}$ by inserting a $(q-1)$-sphere in place of $x$. The definition is
local in character. Suppose $r:V^{m+q}\rightarrow\mathbb{R}^{m}$ is a TRN of
$\mathbb{R}^{m}$ $(\mathbb{R}_{+}^{m}$ in the with-boundary case). Let $U$
$\approx\mathbb{R}^{m}\times2B^{q}$ be a neighborhood of $\mathbb{R}%
^{m}=\mathbb{R}^{m}\times0$ in $\mathring{V}$ such that $U$ is closed and
collared in $V$. Let $\lambda:2B^{q}\rightarrow2B^{q}$ be the map
$\lambda(B^{q})=0$, $\lambda|_{\partial2B^{q}}=\operatorname{id}$ and
$\lambda$ extended linearly on radial lines joining $\partial B^{q}$ to
$\partial2B^{q}$ and define $p:V\rightarrow V$ by letting $p|_{U}%
=\operatorname{id}_{\mathbb{R}^{m}}\times\lambda$ and $p|_{V-U}=$ identity.
Define $r^{\prime}=rp:V\rightarrow\mathbb{R}^{m}$. Let $F_{x}^{\prime}$
denote
\[
(r^{\prime})^{-1}(x)=p^{-1}(F_{x}-x)\cup(x\times B^{q})
\]
with distinguished subsets $\dot{F}_{x}^{\prime}=\dot{F}_{x}$, $(D_{x},\dot
{D}_{x})=x\times(B^{q},\partial B^{q})$ and $A_{x}=F_{x}^{\prime
}-\operatorname*{int}D_{x}$.%
\begin{figure}[th]
\begin{center}
\includegraphics{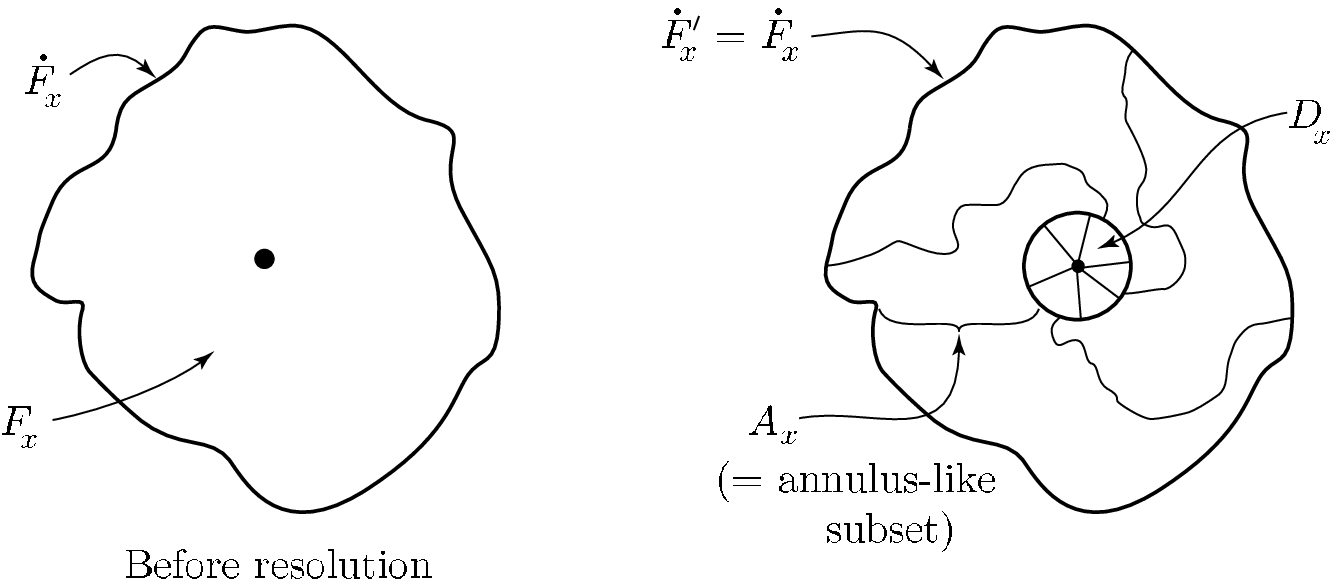}
\end{center}
\end{figure}

It is another straightforward exercise to deduce parts (1) and (2) of
Proposition 3.1 above from part (2) of the following Proposition, by applying
it to successive coordinate charts of $M$ to manufacture the desired deformations.

\begin{proposition}
Suppose $r:V^{m+q}\rightarrow\mathbb{R}^{m}$ is a cone-like TRN, resolved to
$r^{\prime}:V\rightarrow\mathbb{R}^{m}$ as above. Then

\begin{enumerate}
\item for each $x\in\mathbb{R}^{m}$, the inclusions $\dot{F}_{x}%
\hookrightarrow A_{x}$ and $\dot{D}_{x}\hookrightarrow A_{x}$ are shape
equivalences, and

\item the inclusions $\dot{V}\hookrightarrow V-\mathring{U}$ and $\dot
{U}\hookrightarrow V-\mathring{U}$ are proper homotopy equivalences. In fact,
given any majorant map $\epsilon:\mathbb{R}^{m}\rightarrow(0,\infty)$, there
exist deformation retractions $f_{t}:V-\mathring{U}\rightarrow V-\mathring
{U}\;(rel\;\dot{U})$ of $V-\mathring{U}$ into $\dot{U}$, and $g_{t}%
:V-\mathring{U}\rightarrow V-\mathring{U}\;(rel\;\dot{V})$ of $V-\mathring{U}$
into $\dot{V}$, such that both $r^{\prime}f_{t}$ and $r^{\prime}g_{t}$ are
$\epsilon$-close to $r^{\prime}$.
\end{enumerate}
\end{proposition}

\textbf{Note.} The proof shows that the Proposition holds under the \emph{a
priori} weaker hypothesis that $r$ be CS and each inclusion $\dot{F}%
_{x}\hookrightarrow F_{x}-x$ be a shape equivalence. It also holds if $q\geq
3$, $r$ is CE, and each $\dot{F}_{x}$ is $1$-UV.

\begin{proof}
[Proof of Proposition]Part (1). By the hypotheses and elementary shape theory,
$\dot{V}$ is a strong deformation retract of the noncompact $V-U$ in the
$\epsilon$-controlled manner suggested by part (2) (see below). For each $x$,
this provides a shape map from $A_{x}$ to $\dot{F}_{x}$: just push $A_{x}$
into $V-U$, and homotope it out to $\dot{V}$, as close as desired to $\dot
{F}_{x}$. This is a shape equivalence, the inverse of $\dot{F}_{x}%
\hookrightarrow A_{x}$.

Assuming $r$ is conelike, that is, each $(A_{x}-\dot{D}_{x},\dot{F}_{x})$ is
proper shape equivalent to $(\dot{F}_{x}\times\lbrack0,1),\dot{F}_{x}\times
0)$, then in fact $(V-U,\dot{V})$ is proper homotopy equivalent to $(\dot
{V}\times\lbrack0,1),\dot{V}\times0)$ by a well controlled homotopy, and this
can be used to show each $\dot{D}_{x}\hookrightarrow A_{x}$ is a shape
equivalence, as above.

Consider now the weaker hypothesis of the Note. By excision, each inclusion
$\dot{D}_{x}\hookrightarrow A_{x}$ is degree 1 on $\check{C}$ech cohomology,
and by hypothesis $\dot{F}_{x}$ hence $A_{x}$ has the shape of some sphere,
necessarily $S^{q-1}$. Hence $\dot{D}_{x}\hookrightarrow A_{x}$ is a shape
equivalence. Note that this argument fails when it is not known that $\dot
{F}_{x}$ has the shape of a sphere (c.f. Example 4).

Part (2). Assuming $r$ is cone-like, then part (2) is a quick consequence of
the $\epsilon$-controlled proper homotopy equivalence $(V-U,\dot{V})\sim
(\dot{V}\times\lbrack0,1),\dot{V}\times0)$ mentioned in the second paragraph
above, and in fact there is no need to prove part (1). On the other hand, if
using the hypothesis of the Note, then one wants to know part (1)
$\Rightarrow$ part (2). This implication is a corollary of a Whitehead-type
theorem for shape, which we state in the Appendix.
\end{proof}

\section{Cone-like TRN's are mapping cylinder TRN's}

The purpose of this section is to prove the equivalence of the two definitions
given in \S 1. As already noted, a mapping cylinder TRN is clearly a cone-like TRN.

\begin{theorem}
Suppose $(V^{m+q},M^{m},r_{0})$ is a cone-like topological regular
neighborhood. Suppose $m+q\geq6$, or that $m+q=5$ and the conclusion below
already holds for $(\delta V,\partial M,r_{0}|)$. Then there is a mapping
cylinder retraction $r_{1}:V\rightarrow M$, arbitrarily close to $r_{0}$, such
that $r_{1}^{-1}(\partial M)=r_{0}^{-1}(\partial M)=\delta V$ and
$r_{1}|_{\delta V}=r_{0}|_{\delta V}$ if $r_{0}|_{\delta V}$ is already a
mapping cylinder retraction. Hence $(V,M,r_{1})$ is a mapping cylinder TRN. In
addition there is an arbitrarily small homotopy of cone-like retractions
$r_{t}:V\rightarrow M$, $t\in\lbrack0,1]$, joining $r_{0}$ to $r_{1}$, such
that $r_{t}^{-1}(\partial M)=\delta V$ and $r_{t}|_{\delta V}=r_{0}|_{\delta
V}$ if $r_{0}|_{\delta V}$ is already a mapping cylinder retraction.
\end{theorem}

\begin{proof}
This is proved using radial engulfing (PL if desired) to effect a shrinking
argument, just as in Edwards-Glaser \cite{EG}. The homotopy comes for free.
$\underline{\quad\quad\quad}$.
\end{proof}

\section{The Handle Straightening Theorem and Lemma}

The Existence-Uniqueness Theorem is based on the following Handle
Straightening Theorem, which is inspired by Siebenmann's Main Theorem in
\cite{Si1}. In essence, it is gotten by crossing the source manifold in
Siebenmann's theorem with $B^{q}$.

\indent Recall the notation $f:X{\small
{\includegraphics[
height=0.1064in,
width=0.2093in
]%
{unto.eps}%
}%
}Y$ means that domain $f$ is a subset of $X$.

\begin{theorem}
[\textbf{Handle Straightening Theorem}]\label{handle straightening th}Suppose
given a cone-like TRN $(V^{m+q},B^{k}\times\mathbb{R}^{n},r)$, $k+n=m$,
$m+q\geq6$, along with an open embedding $f:B^{k}\times\mathbb{R}^{n}\times
B^{q}{\small
{\includegraphics[
height=0.1064in,
width=0.2093in
]%
{unto.eps}%
}%
}V$ defined near
\[
\vdash\!\dashv\;\equiv B^{k}\times\mathbb{R}^{n}\times0\cup\partial
B^{k}\times\mathbb{R}^{n}\times B^{q}%
\]
such that $f(x,0)=x$ for
\[
x\in B^{k}\times\mathbb{R}^{n};\ f(\partial B^{k}\times\mathbb{R}^{n}\times
B^{q})=\delta V\equiv r^{-1}(\partial B^{k}\times\mathbb{R}^{n})
\]
and $rf=\operatorname{projection}$ on $\partial B^{k}\times\mathbb{R}%
^{n}\times B^{q}$.

Then there exists a triangle of maps

\begin {diagram}
B^{k}\times\mathbb{R}^{n}\times B^{q}  & &         &   \\
\dTo^F_\approx                    & \rdTo^R  &                          &\\
&          &B^{k}\times\mathbb{R}^{n} &\\
&\ruTo_r    &                         &\\
V^{m+q}                               &          &                          &\\
\end{diagram}

\[
\text{(not commutative)\medskip}%
\]

such that

\begin{enumerate}
\item $R$ is a cone-like TRN retraction to $B^{k}\times\mathbb{R}^{n}%
=B^{k}\times\mathbb{R}^{n}\times0$, with $R^{-1}(\partial B^{k}\times
\mathbb{R}^{n})=\partial B^{k}\times\mathbb{R}^{n}\times B^{q}$,

\item $F$ is a homeomorphism such that $F=f$ near $\vdash\!\dashv$,

\item $R=rF$ over $B^{k}\times(\mathbb{R}^{n}-4\mathring{B}^{n})\cup\partial
B^{k}\times\mathbb{R}^{n}$, and

\item $R=\operatorname{projection}$ over $B^{k}\times B^{n}\cup\partial
B^{k}\times\mathbb{R}^{n}$.
\end{enumerate}
\end{theorem}

\begin{remark}
If we define $r^{\prime}=RF^{-1}$, then

\begin{enumerate}
\item[a)] $r^{\prime}$ is a $q$-disc fiber bundle projection over $B^{k}\times
B^{n}\cup\partial B^{k}\times\mathbb{R}^{n}$, and

\item[b)] $r^{\prime}=r$ over $B^{k}\times(\mathbb{R}^{n}-4\mathring{B}%
^{n})\cup\partial B^{k}\times\mathbb{R}^{n}$.
\end{enumerate}
\end{remark}

\noindent\textbf{Note. }There is a cone-like homotopy joining $r$ to
$r^{\prime}$, but its existence is not immediate from the proof below. The
discussion of such homotopies is deferred until \S .$\underline{\quad\quad}$.

\indent The Theorem above is deduced from the following Lemma using the
inversion device introduced in \cite{Si1}.

\begin{lemma}
[\textbf{Handle Straightening Lemma}]The same data is given, and the same
conclusion is drawn, except that (3) and (4) are replaced by

\begin{enumerate}
\item[(3$^{\prime}$)] $R=rF$ over $B^{k}\times B^{n}\cup\partial B^{k}%
\times\mathbb{R}^{n}$

\item[(4$^{\prime}$)] $R=$ standard projection over $B^{k}\times
(\mathbb{R}^{n}-4\mathring{B}^{n})\cup\partial B^{k}\times\mathbb{R}^{n}$.
\end{enumerate}
\end{lemma}

\begin{proof}
[Proof that Lemma implies Theorem]In this proof, the Handle Lemma is applied
twice, the first time only to compactify $V$.

Let $S^{n}=\mathbb{R}^{n}\cup\infty$. The $F$ and $R$ given by the Handle
Lemma provide, via compactification (see Example \ref{Ex: capping off}), the
$F_{\infty}$ and $R_{\infty}$ in the triangle

\begin {diagram}
B^{k}\times S^{n}\times B^{q}  & &         &   \\
\dTo^{F_{\infty}}_\approx      &\rdTo^{R_{\infty}} &               & \\
&                   &B^{k}\times\mathbb{R}^{n} &    \\
&\ruTo_{r_{\infty}} &               &\\
V_{\infty}                     &                   &               &\\
\end{diagram}

\[
\text{(not commutative)\medskip}%
\]

(The replacement $A\rightsquigarrow A_{\infty}$ for $A=$ any of: $V,F,R,$ or
$\vdash\!\dashv$, suggests compactification, while $A^{\#}$ below suggests the
analogue of $A$ in the inverted context.) Restrict $F_{\infty}$ to a
neighborhood of
\[
\vdash\!\dashv^{\#}\equiv B^{k}\times(S^{n}-0)\times0\cup\partial B^{k}%
\times(S^{n}-0)\times B^{q}%
\]
in
\[
B^{k}\times(S^{n}-0)\times B^{q}%
\]
to get
\[
f^{\#}:B^{k}\times(S^{n}-0)\times B^{q}{\small
{\includegraphics[
height=0.1064in,
width=0.2093in
]%
{unto.eps}%
}%
}V^{\#}\equiv V_{\infty}-r^{-1}(B^{k}\times0).
\]

The Handle Lemma can be applied to TRN's of $B^{k}\times(S^{n}-0)$ by
imagining $S^{n}-0$ identified with $\mathbb{R}^{n}$ by the natural inversion
homeomorphism
\[
\theta:\mathbb{R}^{n}\cup\infty\rightarrow\mathbb{R}^{n}\cup\infty
\]
given by
\[
\theta(y)=y/|y|^{2}\quad\mbox{for}\quad y\neq0,\infty\quad\mbox{and}\quad
\theta(0)=\infty\quad\mbox{and}\quad\theta(\infty)=0.
\]
In such inverted applications, the original subsets $B^{k}\times rB^{n}$ and
$B^{k}\times(\mathbb{R}^{n}-r\mathring{B}^{n})$ of $B^{k}\times\mathbb{R}^{n}$
are replaced by $B^{k}\times(S^{n}-(1/r)\mathring{B}^{n})$ and $B^{k}%
\times((1/r)B^{n}-0)$ of $B^{k}\times(S^{n}-0)$. (Note: Under this
interpretation of inversion, the homeomorphism $\theta$ does \textbf{not}
explicitly appear anywhere in the following proof).

Apply the Handle Lemma to the TRN $r^{\#}\equiv r_{\infty}|:V^{\#}\rightarrow
B^{k}\times(S^{n}-0)$ to get maps $F^{\#}$ and $R^{\#}$ in the triangle

\begin{diagram}
V_{\infty}-r^{-1}(B^{k}\times0)\equiv V^{\#} &  & \\
&\rdTo^{r^{\#}\equiv r_{\infty}|V^{\#}} &\\
\uTo^{F^{\#}}_\approx &  & B^{k}\times(S^{n}-0)\\
&\ruTo_{R^{\#}}   \\
B^{k}\times(S^{n}-0)\times B^{q} &  \\
\end{diagram}

\[
\text{(not commutative)\medskip}%
\]

Thus

(1) $R^{\#}$ is a cone-like TRN retraction to $B^{k}\times(S^{n}-0)$ with
$(R^{\#})^{-1}(\partial B^{k}\times(S^{n}-0))=\partial B^{k}\times
(S^{n}-0)\times B^{q}$,

(2) $F^{\#}$ is a homeomorphism such that $F^{\#}=f^{\#}$ near $\vdash
\!\dashv^{\#}$

(3) $R^{\#}=r^{\#}F^{\#}$ over $B^{k}\times(S^{n}-\mathring{B}^{n}%
)\cup\partial B^{k}\times(S^{n}-0)$, and

(4) $R^{\#}=$ standard projection over $B^{k}\times((1/4)B^{n}-0)\cup\partial
B^{k}\times(S^{n}-0)$.

Extend $r^{\#}$ and $R^{\#}$ using $r_{\infty}$ and $R_{\infty}$ to get

\begin{diagram}
V^{\#} & \rInto  & V_{\infty} &   & \\
& & &\rdTo^{r_{\infty}^{\#}} &\\
\uTo^{F^{\#}}  &   & \includegraphics{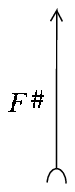} &  &B^{k}\times S^{n}\\
& & & \ruTo_{R_{\infty}^{\#}} &\\
B^{k}\times(S^{n}-0)\times B^{q} &\rInto & B^{k}\times S^{n}\times B^{q}& &\\
\end{diagram}

\[
\text{(not commutative)\medskip}%
\]

We must extend $F^{\#}$ to a homeomorphism $F_{\infty}^{\#}:B^{k}\times
S^{n}\times B^{q}\rightarrow V_{\infty}$. First restrict $F^{\#}$ to
$B^{k}\times(S^{n}-(1/10)\mathring{B}^{k})\times B^{q}$ and then extend over
\[
D\equiv(B^{k}-(1-\epsilon)\mathring{B}^{k})\times(1/10)B^{n}\times B^{q}\cup
B^{k}\times(1/10)B^{n}\times\epsilon B^{q}%
\]
via $f$ (for some small $\epsilon>0)$ to get an embedding
\[
G^{\#}:B^{k}\times S^{n}\times B^{q}-(1-\epsilon)\mathring{B}^{k}%
\times(1/10)\mathring{B}^{n}\times(B^{q}-\epsilon\mathring{B}^{q})\rightarrow
V_{\infty}.
\]
Now the difference $\operatorname*{cl}(V_{\infty}-\operatorname*{image}%
(G^{\#}))$ is a compact $s$-cobordism between manifolds-with-boundary
\[
G^{\#}((1-\epsilon)B^{k}\times(1/10)B^{n}\times\epsilon\mathring{B}^{q})
\]
and $\operatorname*{cl}(\partial V_{\infty}^{\#}-\operatorname*{image}G^{\#}%
)$, with product boundary cobordism.
\[
G^{\#}(\partial\lbrack(1-\epsilon)B^{k}\times(1/10)B^{n}]\times(B^{q}%
-\epsilon\mathring{B}^{q})).
\]
Hence this difference is a product, so $G^{\#}$ extends to $F_{\infty}^{\#}$
as desired.

Finally, taking restrictions to the original sets $V$, $B^{k}\times
\mathbb{R}^{n}\times B^{q}$ and $B^{k}\times\mathbb{R}^{n}$ yields the triangle

\begin{diagram}
V_{\infty} & \lInto & V & & &\\
&&&\rdTo^r&\\
\uTo^{F_{\infty}^{\#}}  &  & \uTo^{F_{\infty}^{\#}|\equiv F_{1}}_\approx &  & B^{k}\times\mathbb{R}^{n}\\
&&&\ruTo_{R_{1}\equiv R_{\infty}^{\#}|} &\\
B^{k}\times S^{n}\times B^{q} &\lInto & B^{k}\times\mathbb{R}^{n}\times
B^{q} & & \\
\end{diagram}

\[
\text{(not commutative)\medskip}%
\]

The maps $F_{1}$ and $R_{1}$ satisfy properties (1) and (2) of the Handle
Theorem, along with

(3$^{\prime\prime}$) $R_{1}=rF_{1}$ over $B^{k}\times(\mathbb{R}^{n}%
-\mathring{B}^{n})\cup\partial B^{k}\times\mathbb{R}^{n}$ and

(4$^{\prime\prime}$) $R_{1}=$ standard projection over $B^{k}\times
(1/4)B^{n}\cup\partial B^{k}\times\mathbb{R}^{n}$.

These are clearly equivalent to (3) and (4) of the Handle Theorem completing
the proof that the Handle Straightening Lemma implies the Handle Straightening Theorem.
\end{proof}

\begin{proof}
[Proof of Handle Straightening Lemma]The proof is based on a diagram which
derives from the classic diagram of Kirby-Siebenmann; its immediate
predecessor is the diagram in \cite{Si1}.

To make certain constructions precise, we make two preliminary modifications
in the given data. First, by compression toward $\vdash\!\dashv$ in
$B^{k}\times\mathbb{R}^{n}\times B^{q}$, we arrange that $f$ is defined on a
neighborhood of $(B^{k}-(1/2)\mathring{B}^{k})\times\mathbb{R}^{n}\times
B^{q}\cup B^{k}\times\mathbb{R}^{n}\times(1/2)B^{q}$ in $B^{k}\times
\mathbb{R}^{n}\times B^{q}$. Second, by redefining $r$ over $B^{k}%
\times4\mathring{B}^{n}$ by conjugation, we arrange that $rf$ is standard
projection over $(B^{k}-(1/2)\mathring{B}^{k})\times3B^{n}\cup\partial
B^{k}\times\mathbb{R}^{n}$. Clearly there is no loss in proving the Lemma for
these modified $r$ and $f$.

The diagram is constructed essentially from the bottom up. All the right hand
triangles commute, as do all the squares but two: the one below $h$ and the
one containing $F$. The details of the construction follow

\begin{diagram}
\phantom{.} &\lTo &B^{k}\times\mathbb{R}^{n}\times B^{q} & & & \rTo^{R \phantom{long gap}}  &
B^{k}\times\mathbb{R}^{n} & &&\\
\dTo & &\uTo_{j\times\operatorname{id}_{B^{q}}}&  &  &  &\uTo_j & &\luInto(3,5) &\\
& &B^{k}\times\mathbb{R}^{n}\times B^{q} &  & &\rTo^{S \phantom{long gap}} &
B^{k}\times\mathbb{R}^{n} & & &\\
& &\dTo_{e\times\operatorname{id}_{B^{q}}} &  &  & &\dTo_e & &\luInto(2,3)&\\
& &B^{k}\times T^{n}\times B^{q} &  & &\rTo^{s \phantom{long gap}} &
B^{k}\times T^{n} & & &\\
& &\uInto &  &  & &\uInto& \luInto(3,1) & & &B^{k}\times2B^{n}\\
& &B^{k}\times T^{n}\times B^{q} &\rTo_\approx^h &W_{1} & \rTo^{r_{1}}
& B^{k}\times T^{n} & & \ldInto(2,1)\ldInto(2,3) &\ldInto(3,5)\\
&&\uInto & &\uInto  & &\uInto & & &\\
& & B^{k}\times T_{0}^{n}\times B^{q} & \includegraphics{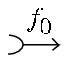}  &W_{0} & \rTo^{r_{0}}
& B^{k}\times T_{0}^{n} & & &\\
& &\dTo_{\alpha\times\operatorname{id}_{B^{q}}} & & \dTo_{\alpha_{0}} &
& \dTo^{\alpha=}_{\operatorname{id}\times \alpha '} & & &\\
& &B^{k}\times\mathbb{R}^{n}\times B^{q} & \includegraphics{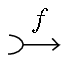} & V^{m+q} & \rTo^r &
B^{k}\times\mathbb{R}^{n} & &&\\
&&&&\uTo&&&&\\
&\rTo^{\phantom{really long gap} F}&&&\phantom{.}&&&&\\
\end{diagram}\medskip

\centerline{\bf Main Diagram}\medskip

\noindent\textbf{Note. }Details in the remainder of the proof are not yet
completely filled in.\newline\newline\textbf{[About }$e$ \textbf{and}
$p$\textbf{.]} Regard $T^{n}$ as the quotient $\mathbb{R}^{n}/(8\mathbb{Z}%
)^{n}$ of $\mathbb{R}^{n}$ where, $\mathbb{Z}$ denotes the integers, and let
$e^{\prime}:\mathbb{R}^{n}\rightarrow T^{n}$ be the corresponding quotient
map. Define $e=\operatorname{id}_{B^{k}}\times e^{\prime}$. Abusively we
regard $B^{k}\times rB^{n}\subset B^{k}\times T^{n}$ for $r<4$. Choose $p\in
T^{n}-2B^{n}$ and let $T_{0}^{n}=T^{n}-p$.\newline\newline\textbf{[About
}$\alpha:B^{k}\times T_{0}^{n}\rightarrow B^{k}\times3\mathring{B}^{n}%
$\textbf{.]} Let $\alpha^{\prime}:T_{0}^{n}\rightarrow3\mathring{B}$ be an
immersion such that $\alpha^{\prime}|_{2B^{n}}=\operatorname{id}$. Define
$\alpha=\operatorname{id}_{B^{k}}\times\alpha^{\prime}$. This makes the four
triangles commute.\newline\newline\textbf{[About }$j:B^{k}\times\mathbb{R}%
^{n}\rightarrow B^{k}\times\mathbb{R}^{n}$\textbf{.]} It is the non-surjective
embedding obtained by restriction of the homeomorphism $J:\mathbb{R}%
^{m}\rightarrow4\mathring{B}^{m}=4\mathring{B}^{k}\times4\mathring{B}^{q}$
which fixes precisely $2\mathring{B}^{m}$ and on each ray from the origin is
linearly conjugate to the homeomorphism $\gamma:[0,\infty)\rightarrow
\lbrack0,-)$ defined by $\gamma|_{[0,-]}=\operatorname{id}$ and $\gamma
(x)=\underline{\quad\quad\quad}$. \newline\newline\textbf{[About }$W_{0}$,
$r_{0}$, $\alpha_{0}$ and $f_{0}$\textbf{.]} These are defined via pullback.
Thus
\[
W_{0}=\{(x,y)\in V\times B^{k}\times(T^{n}-p)\mid r(x)=\alpha(y)\}
\]
and $\alpha_{0}(x,y)=x$ and $r_{0}(x,y)=y$ and $f_{0}\equiv(f|,rf|):B^{k}%
\times(T^{n}-p)\times0\cup\underline{\quad\quad\quad}\rightarrow W_{0}$. We
have that $\alpha_{0}$ is an immersion, $W_{0}$ is a manifold and $r_{0}$ is a
cone-like retraction to $f_{0}(B^{k}\times T_{0}^{n}\times0)$, by the obvious
generalization of \cite[Lemma 2.3]{Si1}. Also, $f_{0}$ is naturally an open
embedding of some neighborhood of $B^{k}\times T_{0}^{n}\times0\cup
(B^{k}-(1/2)\mathring{B}^{k})\times T_{0}^{n}\times B^{q}$, and $r_{0}f_{0}$
is standard projection on this set. \newline\newline\textbf{[Construction of
}$W_{1}$, $r_{1}$ and $h$\textbf{.]} The open embedding
\[
f_{0}|:(B^{k}-(1/2)B^{k})\times T_{0}^{n}\times B^{q}\rightarrow W_{0}%
\]
defines by attachment a manifold
\[
W_{1}^{\prime}\equiv(B^{k}-(1/2)B^{k})\times T^{n}\times B^{q}\cup_{f_{0}%
|}W_{0}%
\]
and an open embedding
\[
f_{1}^{\prime}:\operatorname*{domain}f_{0}\cup(B^{k}-(1/2)B^{k})\times
T^{n}\times B^{q}\rightarrow W_{1}^{\prime}.
\]
Now use infinite $s$-cobordism theorem and capping off (Example
\ref{Ex: capping off}) to get $W_{1},r_{1}$ and $f_{1}$ and then get $\dot{h}$
by the compact $s$-cobordism theorem [Details to be filled out here].\newline%
\newline\textbf{[Construction of }$s$\textbf{.]} The preceding step produced a
conelike retraction
\[
r_{1}h:(B^{k}\times T^{n}-(1/2)B^{k}\times p)\times B^{q}\rightarrow
B^{k}\times T^{n}-(1/2)B^{k}\times p
\]
which is standard projection near $\partial B^{k}\times T^{n}\times B^{q}$.
Let $s$ be the natural compactification of
\[
\omega(r_{1}h)(\omega^{-1}\times\operatorname{id}_{B^{q}})):(B^{k}\times
T^{n}-(0,p))\times B^{q}\rightarrow B^{k}\times T^{n}-(0,p)
\]
where
\[
\omega:B^{k}\times T^{n}-(1/2)B^{k}\times p\rightarrow B^{k}\times
T^{n}-(0,p)
\]
is a homeomorphism which is fixed near \underline{$\quad\quad\quad$}%
.\newline\newline\textbf{[Construction of }$S$ and $R$\textbf{.]} $S$ is the
unique covering cone-like retraction. $R$ is defined by $jS(j\times
\operatorname{id})^{-1}$ on $j(B^{k}\times\mathbb{R}^{n})\times B^{q}$, and is
extended via the identity over all of $B^{k}\times\mathbb{R}^{n}\times B^{q}$.
It is the crux of the torus device that $R$ is continuous. \newline%
\newline\textbf{[Construction of }$F$\textbf{.]} The left hand side of the
diagram from top to bottom defines an open embedding
\[
\phi:B^{k}\times2\dot{B}^{n}\times B^{q}\rightarrow r^{-1}(B^{k}\times2\dot
{B}^{n})\subset V
\]
such that $r\phi=R|_{B^{k}}\times2\dot{B}^{n}\times B^{q}$. Extend $\phi$ over
\underline{$\quad\quad\quad$} by $f$ and then over all of $B^{k}%
\times\mathbb{R}^{n}\times B^{q}$ by engulfing, to get $F$. All the necessary
homotopies for engulfing follow from the Homotopy Proposition (Prop.
\ref{Prop (homotopy proposition)}); recall the engulfing may be PL if desired,
as $\operatorname*{int}V$ is PL triangulable.
\end{proof}

\section{Proof of Existence-Uniqueness Theorem}

This section proves Theorem \ref{existence-uniqueness thm}, without the
cone-like homotopy, but with the well-controlled ambient isotopy.

\begin{proof}
[Sketch of Proof]Existence follows from a good uniqueness theorem;
\textquotedblleft good\textquotedblright\ means we want a relative C-D
statement as in \cite[p.71]{EK}. This good uniqueness theorem follows in
straightforward fashion from the Handle Straightening Theorem
\ref{handle straightening th}, much like the situation in \cite{EK}.
\end{proof}

\section{Local contractibility of the space of cone-like retractions;
cone-like homotopies}

This section is independent of the preceding Sections 2-6, and has no
dimension restrictions. This section plays a role in this paper analogous to
the role of the local contractibility of the homeomorphism group of a manifold
(\cite{Ce}, \cite{EK}) in Siebenmann's paper \cite{Si1}.

\indent Let $M$ be a fixed manifold and $V$ a fixed topological regular
neighborhood of $M$, with distinguished submanifold $\delta V\subset\partial
V$ but \emph{without} a specific retraction. Let $C(V,M)$ be the space of all
cone-like retractions $r:V\rightarrow M$ such that $r^{-1}(\partial M)=\delta
V$, topologized with the majorant topology given by majorant maps on $M$. That
is, given majorant map $\epsilon:M\rightarrow(0,\infty)$, the $\epsilon
$-neighborhood of $r:V\rightarrow M$ is
\[
N(r,\varepsilon)=\{p\in C(V,M)\mid d(p(x),r(x))<\epsilon
(r(x))\;\mbox{for all}\;x\in V\}
\]
where $d$ is the metric on $M$. Although $C(V,M)$ is decidedly non-metric if
$M$ is not compact, it turns out that $C(V,M)$ is closed under Cauchy limits
if $d$ is a complete metric (Compare \cite{Si1}); this fact is not so
essential to us as the following facts.

\indent Call a cone-like retraction $r:V\rightarrow M$ \emph{locally
approximable by bundle }\newline\emph{projections} (\emph{locally
approximable} for short) if each $x\in M$ has an open neighborhood $W$ in $M$
such that $r|:r^{-1}(W)\rightarrow M$ is arbitrarily closely approximable by
disc bundle projections (uniformly, not majorantly). Let $C_{0}(V,M)$ denote
the subset of $C(V,M)$ of all such locally approximable retractions. Of
course, it is a corollary of Section 6 that $C_{0}(V,M)=C(V,M)$ if $\dim
V\geq6$; however, working with $C_{0}(V,M)$ obviates dimension restrictions.

\indent The goal of this section is to show that $C_{0}(V,M)$ is locally
0-connected (defined below) and that a certain cone-like homotopy extension
principle holds, analogous to the isotopy extension principle for
homeomorphisms. Actually $C_{0}(V,M)$ is locally $k$-connected for all $k$ by
a routine adaptation of the Eilenberg-Wilder argument. The torus techniques
for local contractibility fail us in this section so we turn to an adaptation
$\underline{\quad\quad\quad}$

\begin{proposition}
\textbf{(1)} Suppose $r\in C_{0}(V,M)$ and $U\approx\mathbb{R}^{m}$ or
$U\approx\mathbb{R}_{+}^{m}$ is a coordinate chart in $M$. Then $r|_{r^{-1}%
(U)}$ is arbitrarily closely approximable by disc bundle projections
(uniformly here).

\noindent\textbf{(2)} $C_{0}(V,M)$ is closed in $C(V,M)$. (As noted above
$C(V,M)$ is closed in $P(V,M)=$ all proper maps $V\rightarrow M$, but we don't
need this).
\end{proposition}

\begin{theorem}
The local 0-connectivity (indeed locally $k$-connectivity) of $C_{0}(V,M)$, as
mentioned above.
\end{theorem}

\begin{proof}
This proof is accomplished by a limit argument, by first proving the
\textquotedblleft almost\textquotedblright\ local contractibility of the space
of disc bundle projections. This is completely analogous to my proof by
\textquotedblleft almost handle straightening\textquotedblright, using
\v{C}ernavskii meshing, of the following result.\medskip
\end{proof}

\begin{theorem}
\textbf{(A)} Given any PL manifold $M$ and $\epsilon>0$, there exists
$\delta>0$ such that if $h:M\rightarrow M$ is a PL homeomorphism which is
$\delta$-close to the identity, then $h$ may be PL $\epsilon$-isotoped as
close as desired to $\operatorname{id}_{M}$ (but not to $\operatorname{id}%
_{M}$ by counterexample of Kirby-Siebennman.). This process is canonical
PL.\newline\textbf{(B) }(from \textbf{(A).}) The PL homeomorphism group of $M$
is locally contractible as a topological group (but not as a semisimplicial complex).
\end{theorem}

\section{Topological regular neighborhoods of polyhedra in manifolds}

This section sketches the extension of the previous sections to polyhedra in
manifolds. There are two technical points that have to be sorted out before
saying that the definitions of TRN's routinely extend to polyhedra. The first
concern is what is the polyhedral analogue of locally flat. The second
concerns what to do at the boundary $\delta V$, for polyhedral pairs.

\indent A faithful PL embedding $f:(X,Y)\rightarrow(Q,\partial Q)$ of a
polyhedral pair into a PL manifold is \emph{locally homotopically unknotted}
if for each $x\in X$, both deleted links
\[
\operatorname*{lk}(f(x),Q)-\operatorname*{lk}(f(x),f(X))
\]
and (if $x\in Y$)%
\[
\operatorname*{lk}(f(x),\partial Q)-\operatorname*{lk}(f(x),f(Y))
\]
have\emph{ free} $\pi_{1}$ (for each component). For codimension $\geq3$ this
is always true as the $\pi_{1}$'s are trivial by general position. Note that
for $(X,Y)$ a codimension 2 manifold $(M,\partial M)$, this is just the usual
local homotopy unknottedness definition.

\indent A faithful topological embedding $f:(X,Y)\rightarrow(Q,\partial Q)$ of
a polyhedral pair into a topological manifold is \emph{locally tame} if for
each $x\in X$ there is an open neighborhood $(U,\partial U)$ of $f(x)$ in
$(Q,\partial Q)$ such that the embedding $f|:f^{-1}(U,\partial U)\rightarrow
(U,\partial U)$ is PL locally unknotted for some PL manifold structure on
$(U,\partial U)$. Note the PL structure on the source is induced from $X$, but
the PL structures on the $U$'s (for various $x$) need not be compatible. The
unknottedness condition is independent of the PL structure on $U$.

\indent In the definition of TRN's for polyhedra one should replace
\textquotedblleft$(M,\partial M)$ locally flat in $(V,\partial V)$%
\textquotedblright\ with \textquotedblleft$(X,Y)$ locally tame in $(V,\partial
V)$\textquotedblright.

\indent There is another change required in case $Y\neq\varnothing$, because
the condition that $r^{-1}(Y)$ be a TRN of $Y$ is too restrictive, as the
following example shows.

\begin{example}
Let $X$ be an interval, $Y$ the midpoint of $X$, and $\left(  V,\partial
V\right)  =(2$-disc, boundary$)$ as shown. Then $r^{-1}\left(  Y\right)  $
must be disconnected.%
\begin{center}
\includegraphics[
trim=0.000000in 0.000000in -0.056877in 0.000000in,
height=1.5056in,
width=1.5056in
]%
{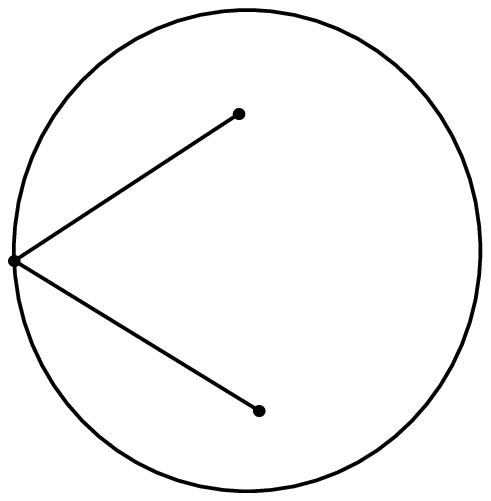}%
\end{center}

\end{example}


\indent This same problem cropped up in \cite{E1} and the remedy is the
same---namely to not require $\delta V$ to be all of $r^{-1}\left(  Y\right)
$. Details are trivial.

\indent Having established these two technical points, then the definition of
TRN's (either mapping cylinder or cone-like) for polyhedra in manifolds is as indicated.

\begin{theorem}
Existence-Uniqueness holds exactly as in the manifold case (with the same
dimension restrictions on $Q$).
\end{theorem}

Perhaps the quickest proof of this is by analogy:%

\[
\frac{\text{This proof}}{\text{Proof of Th. \ref{existence-uniqueness thm}}%
}\quad\approx\quad\frac{\text{Siebenmann's \cite{Si2}}}{\text{Edwards-Kirby's
\cite{EK}}}%
\]
\newline\indent That is, the extension to the above theorem of the proofs in
\S 5 and \S 6 is completely analogous to the extension to locally cone-like
TOP stratified sets of the local contractibility of the homeomorphism group of
a topological manifold, done by Siebenmann.

\indent The local unknottedness hypothesis ensures that the $s$-cobordism
theorem holds at all applications. Details omitted here.

\indent Replace CW complex by cell complex (i.e. don't need skeletal
filtration that CW complexes have.

\section{CW\ complexes}

\begin{remark}
In the following, one may replace `CW complex' with `cell complex'. In
particular, one doesn't need the skeletal filtration preent in CW complexes.
\end{remark}

It turns out the CW complexes in manifolds have topological regular
neighborhoods \emph{stably}, that is, $X\subset Q$ has a topological regular
neighborhood in $Q\times\mathbb{R}^{s}$ for some $s\geq0$. Furthermore they
are unique nonstably. The most useful application of these facts seems to be a
proof that a CE map ($\equiv$ proper cell-like surjection) of CW complexes is
a simple homotopy equivalence (first proved for homeomorphisms by Chapman
\cite{Ch}). Our discussion below is toward this goal.

\indent All our CW complexes from now on are \emph{finite} (i.e., compact).
This discussion trivially generalizes to nonfinite CW complexes of finite
dimension, but we postpone details for arbitrary CW complexes.

\indent Either definition of topological regular neighborhood given at the
start of the paper is valid with $M$ replaced by a CW complex $X$, subject to
certain provisos. For the mapping cylinder definition, they are: regard
$\partial X=\varnothing=\delta V$ always; replace \textquotedblleft locally
flat\textquotedblright\ by \textquotedblleft each $\dot{F}_{x}$ is
$1$-UV\textquotedblright, and always assume $\dim X\leq\dim V-3$. For the
second definition, the provisos are the same, except that \textquotedblleft
locally flat\textquotedblright\ is replaced by \textquotedblleft$X$ is $1$-LCC
in $V$\textquotedblright, that is, $V-X$ is $1$-LC at $X$. This implies each
$\dot{F}_{x}$ is $1$-UV, and in the presence of mapping cylinder structure,
the conditions are equivalent.

\begin{remark}
If $M_{f}$ is a mapping cylinder for some proper map $f:A\rightarrow B$, then
$M_{f}\times I^{k}$ (with $I^{k}=[-1,1]^{k}$) has a natural mapping cylinder
structure for the map
\[
(f\times\pi)|:(A\times I^{k}\cup M_{f}\times\partial I^{k})\rightarrow
B\times0=B
\]
where $\pi:I^{k}\rightarrow0$ is projection, as suggested by the diagram. The
new fibers $\{F_{b}\times I^{k}\}$ have $UV^{k-1}$ boundaries
\[
\{(F_{b}\times I^{k})^{\cdot}=\dot{F}_{b}\times I^{k}\cup F_{b}\times\partial
I^{k}\},
\]
regardless of the nature of $F_{b}$, because $(F_{b}\times I^{k})^{\cdot}$ has
the shape of $\Sigma^{k-1}\dot{F}_{b}$.%
\begin{figure}[th]
\begin{center}
\includegraphics{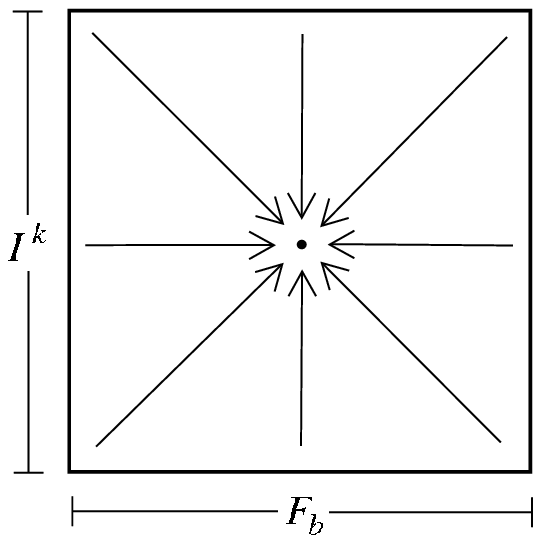}
\end{center}
\end{figure}
\end{remark}

\begin{theorem}
\label{s.h.e. theorem for TRNs}If $V$ is any abstract TRN of CW complex $X$,
as defined above, then $X\hookrightarrow V$ is a simple homotopy equivalence.
\end{theorem}

\begin{theorem}
\label{existence-uniqueness thm for CW complexes}Suppose $X\subset Q$ is a CW
complex embedded in a topological manifold, $\partial Q=\varnothing$. Then

\begin{enumerate}
\item (Existence) $X$ has a mapping cylinder TRN in $Q\times\mathbb{R}^{s}$
for some $s\geq0$, and

\item (Uniqueness) If $V_{0}$ and $V_{1}$ are two TRN's of $X$ in $Q$, then
$V_{0}$ is homeomorphic to $V_{1}$ by ambient isotopy of $Q$ which fixes a
neighborhood of $X$.
\end{enumerate}
\end{theorem}

\begin{corollary}
[to Theorem \ref{s.h.e. theorem for TRNs} and Part 1 of Theorem
\ref{existence-uniqueness thm for CW complexes}]A homeomorphism
$h:X\rightarrow Y$ of CW complexes is a simple homotopy equivalence.

\begin{proof}
[Proof of Corollary]Let $V$ be a TRN of $Y$, by Theorem 2. Then Theorem 1 says
that both the inclusion $\eta:Y\rightarrow V$ and the embedding $\eta
h:X\rightarrow V$ are simple homotopy equivalences, hence so is $h$.
\end{proof}
\end{corollary}

\begin{proof}
[Proof of Theorem \ref{s.h.e. theorem for TRNs}]This is just an extension to
CW complexes of an argument in \cite{E2}. One inducts on the number of cells
in $X$, and uses TRN uniqueness to accomplish the splitting of $V$ over
$S^{n-1}\times0$ in $S^{n-1}\times(-1,1)=$ open collar neighborhood of
$\infty$ in the last open cell of $X$. Once $V$ is split, one applies
induction and the Sum Theorem.
\end{proof}

\begin{proof}
[Proof of Theorem \ref{existence-uniqueness thm for CW complexes}%
]\emph{(Uniqueness)}. Pull $V_{0}$ into $\operatorname*{int}V_{1}$ by
engulfing, and then apply the $s$-cobordism theorem to the difference
$V_{1}-\operatorname{int}V_{0}$. It is an $s$-cobordism because $V_{0}\subset
V_{1}$ is a simple homotopy equivalence, and throwing away
$\operatorname*{int}V_{0}$ with its codimension $\geq3$ spine does not change
this.\medskip

\noindent\emph{(Existence)}\textbf{ }Interestingly, the proof has nothing to
do with the previous theory; it is just a straightforward inductive exercise.

We first remark that in the following construction, the advantage of always
working in the ambient manifold $Q\times\mathbb{R}^{s}$ (even if
$Q=\mathbb{R}^{q}$), rather than constructing $V$ in the abstract, is that it
automatically provides the correct framing for the normal bundle of the
embedding $g_{\partial}:\partial D^{n}\rightarrow\partial(V\times B^{n})$
(defined below) which is used to attach the handle. If one didn't choose this
framing correctly, some future $g_{\partial}$ might not have a framing. Thus,
working in $Q$ obviates paying attention to bundle trivializations.

Suppose $Y$ is a CW complex with mapping cylinder TRN $r:V\rightarrow Y$,
where $V$ is a collared, $\operatorname*{codim}0$ submanifold of $Q$. Suppose
$X=Y\cup_{f|_{\partial D^{n}}}f(D^{n})\subset Q$ where $f:D^{n}\rightarrow Q$
is such that $f(D^{n})\cap Y=f(\partial D^{n})$ and $f|_{\operatorname*{int}%
D^{n}}$ is an embedding. Define $g_{\partial}:\partial D^{n}\rightarrow
\operatorname{int}V\times\partial B^{n}\subset\partial(V\times B^{n})$ by
$g_{\partial}(x)=(f(x),x)$ (recall $D^{n}=B^{n})$; it is a locally flat
embedding. Let
\[
F=\{\lambda g_{\partial}(x)\mid x\in\partial D^{n},0\leq\lambda\leq1\}\subset
V\times B^{n}%
\]
be the submapping cylinder of the natural map $g_{\partial}(\partial
D^{n})\rightarrow f(\partial D^{n})\subset Y$, where the fibers $\{\lambda
_{w}\}$ are those of the natural mapping cylinder retraction $r_{1}:V\times
B^{n}\rightarrow Y$.

Let $g:D^{n}\rightarrow Q\times\mathbb{R}^{n}-\operatorname*{int}(V\times
B^{n})$ be a locally flat embedding extending $g_{\partial}$, such that $g$ is
homotopic to $f$ in $V\times\mathbb{R}^{n}$ by a homotopy which agrees in
$\partial D^{n}$ with the straight line homotopy in $F$ joining $g_{\partial}$
to $f|_{\partial D^{n}}$. Then $X^{\prime}\equiv Y\cup F\cup g(D^{n})$ is
homeomorphic to $X$ by the restriction $h|:X^{\prime}\rightarrow X$ of a
homeomorphism $h:Q\times\mathbb{R}^{n}\rightarrow Q\times\mathbb{R}^{n}$
(since homotopy yields isotopy in the trivial range). Thus it suffices to
construct a TRN $V^{\prime}$ for $X^{\prime}$ in $Q\times\mathbb{R}^{n}$.

Let $(H,\delta H)$ be the total space of a normal disc bundle for
$(g(D^{n}),g(\partial D^{n}))$ in $(Q\times\mathbb{R}^{n}-\operatorname*{int}%
(V\times B^{n}),\partial(V\times B^{n}))$. Then $(H,\delta H)\approx
(g(D^{n}),g(\partial D^{n}))\times B^{q}$. Define $V^{\prime}=V\times
B^{n}\cup_{\delta H}H$. We can define mapping cylinder retraction $r^{\prime
}:V^{\prime}\rightarrow X^{\prime}$ by adjusting the mapping cylinder
retraction $r_{1}:V\times B^{n}\rightarrow Y$ to \textquotedblleft turn the
corner\textquotedblright\ near $F$, and then extending over $H$, as follows.
Let $r_{1}^{\prime}:V\times B^{n}\rightarrow Y\cup F$ be the mapping cylinder
retraction obtained from $r_{1}$ as suggested by the following figure (note
the identifications made on the bottom of the rectangles are compatible with
the indicated projections.)
\begin{figure}[th]
\begin{center}
\includegraphics{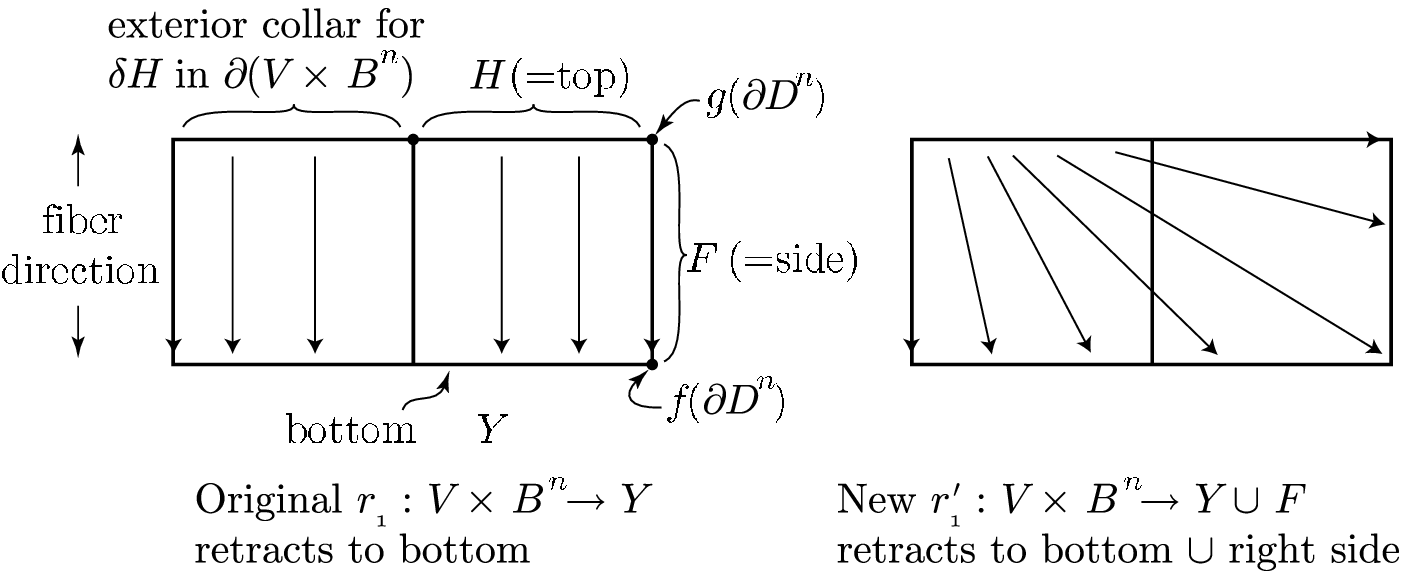}
\end{center}
\end{figure}

In particular, $r_{1}^{\prime}|:\delta H\rightarrow g(\partial
D^{n})$ is standard projection, so $r_{1}^{\prime}$ extends, using standard
projection $H\rightarrow g(D^{n})$, to a retraction $r^{\prime}:V^{\prime
}\rightarrow X^{\prime}$. Clearly $r^{\prime}$ is a mapping cylinder
retraction, and the $1$-UV property follows because $X^{\prime}$ is $1$-LCC in
$V^{\prime}$.

There is an interesting alternative way of defining $r^{\prime}:V^{\prime
}\rightarrow X^{\prime}$, observed by Siebenmann. Let $p_{1}:V^{\prime
}\rightarrow V\times B^{n}\cup g(D^{n})$ be the extension-via-the-identity of
some natural relative mapping cylinder retraction $p_{0}:H\rightarrow\delta
H\cup g(D^{n})$ and let $p_{2}:V\times B^{n}\cup g(D^{n})\rightarrow
X^{\prime}$ (\textbf{not} a retraction but a natural extension of
$r_{1}:V\times B^{n}\rightarrow Y$ such that $p_{2}|:g(\operatorname*{int}%
D^{n})\rightarrow X^{\prime}-Y$ is a homeomorphism. Then $p=p_{2}%
p_{1}:V^{\prime}\rightarrow X^{\prime}$ is a CE map which restricts in
$X^{\prime}$ to a CE map $p|:X^{\prime}\rightarrow X^{\prime}$. In the usual
fashion, let $q:V^{\prime}\rightarrow V^{\prime}$, with $q|_{\partial
V^{\prime}}=\operatorname{id}$, be a map which is a homeomorphism off $F$,
such that $q|_{X^{\prime}}=p|_{X^{\prime}}$. Then $r^{\prime}\equiv
pq^{-1}:V^{\prime}\rightarrow X^{\prime}$ is a well-defined mapping cylinder retraction.
\end{proof}

\section{Concerning mapping cylinder neighborhoods of other compacta}

Consider the following wildly optimistic

\textit{Every compact ENR (= euclidean neighborhood retract) }$X\subset R^{n}%
$\textit{, with }$\dim X\leq n-3$\textit{ and }$R^{n}-X$\textit{ }%
$1$\textit{-LC at }$X$\textit{, has a manifold mapping cylinder neighborhood
which is unique up to homeomorphism. Or at least, every such }$X$\textit{ has
such a unique neighborhood stably, in some }$\mathbb{R}^{n+p}$\textit{. (This
conjecture has a natural Hilbert cube version for compact ANR's).}

This conjecture is stronger than Borsuk's question (the finite dimensional
version) of whether compact ENR's have finite homotopy type; equivalent to
Borsuk's question is whether such $X$ as above have radial neighborhoods in
$\mathbb{R}^{n}$ or even $\mathbb{R}^{n+p}$ (recall $U$ is radial if
$U-X\approx Y\times\mathbb{R}^{1}$ for some compactum. (See \cite{Si3} for
best known implications). Incidentally, the easiest way to prove the
implication: $X$ has finite type $\Rightarrow$ $X$ has a radial neighborhood
stably, is to use the following readily proved stable version of
Geogehan-Summerhill \cite{GS}: two compact subsets $X$ and $Y$ of
$\mathbb{R}^{n}$ have the same (Borsuk) shape $\Leftrightarrow$ the quotients
$\mathbb{R}^{2n+2}/X\approx\mathbb{R}^{2n+2}/Y$ are homeomorphic.

If the conjecture above is true, it would imply that all such $X$ are CE
images of manifolds. It is known conversely that any finite dimensional CE
image of a manifold is an ENR. And such an ENR does have a mapping cylinder
neighborhood stably, namely a quotient of one for the source manifold
stabilized, via the decomposition argument of \cite{Sh}.

\indent It is interesting to compare the Conjecture to two questions raised by
Chapman in the Proceedings of the 1973 Georgia Topology Conference. These are
finite dimensional versions. Let $X$ be a compact ENR and $K,L$ finite cell
complexes.\medskip

\noindent\textbf{Question 1.} \emph{If }$f:K\rightarrow X$\emph{ and
}$g:L\rightarrow X$\emph{ are CE mappings, does there exist a simple homotopy
equivalence }$h:K\rightarrow L$\emph{ such that }$gh\sim f$\emph{?}\medskip

\noindent\textbf{Question 2.} \emph{If }$f:X\rightarrow K$\emph{ and
}$g:X\rightarrow L$\emph{ are CE mappings, does there exist a simple homotopy
equivalence }$h:K\rightarrow L$\emph{ such that }$hf\sim g$\emph{?}\medskip

The answer to Question 1 is yes if the stable \emph{uniqueness} part of the
Conjecture is true; the answer to Question 2 is yes if the stable
\emph{existence} part of the Conjecture is true.

\newpage

\section{\textbf{Appendix: }An extension of some well known homotopy theorems}

This appendix presents a useful generalization of the familiar Whitehead
theorem for weak homotopy equivalences. Using an elementary shape theory
definition, the Theorem encompases Whitehead's Theorem on the one hand
($Z=\operatorname{point}$), and the Lacher-Kozlowski-Price-$\underline
{\quad\quad\quad}$ Theorem for cell-like mappings on the other hand, in
addition to having applications in between.

\indent We work in the category of locally compact metric ANR's and proper
maps (whose point universes need \textbf{not} be ANR's).

\indent A map $f:X\rightarrow Y$ of compact metric spaces (\emph{not}
necessarily ANR's) is a $k$\emph{-shape equivalence}
\label{defn: k-shape equivalence}if both $X$ and $Y$ have finitely many
components and $f$ induces isomorphisms on the homotopy groups up through
dimension $k$. As these homotopy groups are awkward inverse limits, we give
the definition in primitive form (assuming $X$ and $Y$ connected; otherwise
make it hold componentwise). If $X\hookrightarrow L$ and $Y\hookrightarrow M$
are embedded as subsets of ANR's $L$ and $M$ and if $U_{X}$ and $U_{Y}$ are
arbitrary neighborhoods then there are smaller neighborhoods $V_{X}\subset
U_{X}$ and $V_{Y}\subset U_{Y}$ and a map $F:V_{X}\rightarrow V_{Y}$ extending
$P:X\rightarrow Y$, such that for any $i$, $0\leq i\leq k:$

\begin{itemize}
\item \emph{injectivity}: for any map $\alpha:S^{i}\rightarrow V_{X}$ if
$F\alpha\sim0$ in $U_{Y}$, then $\alpha\sim0$ in $U_{X}$, and

\item \emph{surjectivity}: for any map $\beta:S^{i}\rightarrow V_{Y}$, there
is a map $\alpha:S^{i}\rightarrow V_{X}$ such that $F\alpha\sim\beta$ in
$U_{Y}$. Surjectivity can in fact be accomplished by homotopy rel basepoint,
as a consequence of $\pi_{1}$ surjectivity.
\end{itemize}

\indent As usual in shape theory, this definition holds for any pair of
embeddings of $X$ and $Y$ into ANR's if it holds for one pair.

\indent Some authors would define a $k$-shape equivalence as being only
surjective in dimension $k$ (e.g. \cite[p.404]{Sp}, \cite{Ko}) and would prove
the following theorem with $\dim J\leq k$ and $J=K$. However, it seems that
for applications, the form we state it in is perhaps more natural.

If $f:X\rightarrow Y$ is a map and $p:Y\rightarrow Z$ is a surjection, then
$f$ is a $k$\emph{-shape equivalence over} $Z$ if for each $z\in Z$,
$f|:f^{-1}(p^{-1}(z))\rightarrow p^{-1}(z)$ is a $k$-shape equivalence.
\smallskip

\noindent\textbf{Note.} In the following, [proper] means \textquotedblleft
proper\textquotedblright\ is optional. The theorem and corollary are most
believable with proper in place. In fact, on page
\pageref{defn: k-shape equivalence}, I haven't defined $k$-shape equivalent
for non-compact spaces.\smallskip

\begin{theorem}
[{Compare \cite[p.404, Th.22]{Sp} and \cite{Ko}}]\textit{Suppose
}$f:X\rightarrow Y$\textit{ is a [proper] map of locally compact metric ANR's
and }$p:Y\rightarrow Z$\textit{ is a surjection to a separable metric space
}$Z$\textit{. Suppose }$f$\textit{ is a }$k$\textit{-shape equivalence. In the
diagram below, suppose }$J$\textit{ is an arbitrary simplicial complex, }$\dim
J\leq k+1$\textit{, with subcomplex }$L$\textit{, and }$g:L\rightarrow
X$\textit{ and }$h:J\rightarrow Y$\textit{ are maps which make the diagram
commute.}

\textit{Given any majorant map }$\epsilon:Z\rightarrow(0,\infty)$\textit{,
there exists a lift }$g^{\prime}:J\rightarrow X$\textit{ extending }%
$g$\textit{ such that }$pfg^{\prime}$\textit{ is }$\epsilon$\textit{-close to
}$ph$\textit{. Furthermore, if }$K$\textit{ is a subcomplex of }$J$\textit{
with }$\dim K\leq k$\textit{ then }$fg^{\prime}|K$\textit{ may be assumed
}$(p,\epsilon)$\textit{-homotopic to }$h|_{K}$\textit{.}

\begin{diagram}
L & \rTo^g & X \\
\dInto & \ruDotsto^{g'} & \dTo_f \\
J & \rTo^h & Y \\
&  & \dTo_p \\
&  & Z \\
\end{diagram}

\begin{proof}
Standard lifting argument.
\end{proof}
\end{theorem}

The following Corollary encompasses several well-known theorems.

\begin{corollary}
Suppose $f:X\rightarrow Y$ is a [proper] map of locally compact metric ANR's
such that for some $k$,

\begin{enumerate}
\item $\dim X\leq k$ and $\dim Y\leq k$, and

\item $f$ is a $k$-shape equivalence over $Z$ for some proper surjection
$p:Y\rightarrow Z$.
\end{enumerate}

Then $f$ is a [proper] homotopy equivalence. In fact, there is a [proper]
homotopy inverse $g:Y\rightarrow X$ such that $fg\sim\operatorname{id}_{Y}$ by
an arbitrarily $p$-small homotopy, and $gf\sim\operatorname{id}_{X}$ by an
arbitrarily $pf$-small homotopy.\newline\newline\noindent\textbf{Proof.}
Routine mapping cylinder-nerve argument.\vspace{1in}
\end{corollary}

\part{\bigskip}

\section{Additional topics\bigskip}

This part is not yet written. Topics to include: neighborhoods of a pair,
neighborhoods by restriction, Lickorish-Siebenmann Theorem for TRN's,
transversality (with discussion of Hudson's example), the group TOP.

\end{document}